\documentclass[11pt,twoside]{article}
\usepackage{amsmath}
\usepackage{latexsym} 
\usepackage{enumitem}
\usepackage{color}

\input{amssym.def} 
\input{amssym} 
\newfont{\fra}{eufm10 scaled 1095} 
\newfont{\Bb}{msbm10 scaled 1095}
\newfont{\Bbu}{msbm10 scaled 1395}
\newfont{\Bbg}{msbm10 scaled 1680} 
\newcommand\CC{{\mbox{\Bb C}}} 
\newcommand\RR{{\mbox{\Bb R}}}
\newcommand\RRu{{\mbox{\Bbu R}}} 
\newcommand\NN{{\mbox{\Bb N}}} 
\newcommand\ZZ{{\mbox{\Bb Z}}}

\newcommand\TT{{\mbox{\Bb T}}}

\newcommand\fg{{\frak{g}}}

\newcommand\fn{{\frak n}}

\newcommand\cH{{\cal H}}

\newcommand\cS{{\cal S}}
\newcommand\cR{{\cal R}}
\newcommand\cV{{\cal V}}
\newcommand\ph{\varphi}
 
\newcommand\eps{\varepsilon} 

\newcommand{\fsl}{\mathop{{\frak s \frak l}}}

\newcommand{\GL}{\mathop{{\rm GL}}} 
\newcommand{\SL}{\mathop{{\rm SL}}} 
\newcommand{\SO}{\mathop{{\rm SO}}}

\newcommand{\Spin}{\mathop{{\rm Spin}}}

\newcommand{\ad}{{{\rm ad}}}

\newcommand{\del}{{\rm \partial}}
\newcommand{\sgn}{\mathop{{\rm sgn}}}

\newcommand{\Cl}{{{\cal C}l}}

\renewcommand{\Re}{\mathop{{\rm Re}}} 
\renewcommand{\Im}{\mathop{{\rm Im}}}
\newcommand{\diag}{\mathop{{\rm diag}}} 
 
\newcommand{\codim}{{\rm codim}}
\newcommand{\Span}{{{\rm span}}} 
 
\newcommand{\proj}{{{\rm pr}}} 
\newcommand{\GG}{\Gamma\setminus G}
\newcommand{\ind}{{{\rm ind}}}
\newcommand{\supp}{{{\rm supp}}}
\newcommand{\Div}{{\rm div}}
\newcommand{\dom}{{\rm dom}}
\newcommand\ip{{\langle\cdot \,,\cdot \rangle}} 
\newcommand\rip{{(\cdot \,,\cdot )}}

\newcommand\proof{{\sl Proof. }} 
\newcommand{\qed}{\hspace*{\fill}\hbox{$\Box$}\vspace{2ex}} 
\newcommand{\benur}{\begin{enumerate}[label=(\roman*)]}
\newcommand{\la}{\langle}
\newcommand{\ra}{\rangle}
\newtheorem{theo}{Theorem}[section]
\newtheorem{pr}[theo]{Proposition}

\newtheorem{ex}[theo]{Example}
\newtheorem{re}[theo]{Remark}

\newtheorem{lm}[theo]{Lemma}
\begin{document} 

\title{Spectra of sub-Dirac operators on certain nilmanifolds}
\author{Ines Kath and Oliver Ungermann}
\maketitle

\begin{abstract}
We study sub-Dirac operators that are associated with left-invariant bracket-generating
sub-Riemannian structures on compact quotients of nilpotent semi-direct products
$G=\RR^n\rtimes_A\RR$. We will prove that these operators admit an $L^2$-basis of
eigenfunctions. Explicit examples show that the spectrum of these operators can be
non-discrete and that eigenvalues may have infinite multiplicity.
\end{abstract}

\vspace{0.75cm}

\centerline{\textbf{MSC:} 53C17 (35P10, 43A85)}

\section{Introduction} 
Spectra of sub-Laplace operators on sub-Riemannian manifolds are intensely studied.
Especially interesting is the case where the distribution defining the sub-Riemannian
structure is bracket generating, what we shall assume in the following. In this case
the sub-Laplacian is known to be hypo-elliptic \cite{H}. 

Many explicit calculations of the spectrum have been done in the situation where the
underlying manifold is a compact Lie group or a quotient of a Lie group by a discrete
 cocompact subgroup, see, for example, \cite{BF1,BFI,BF2,P}. In \cite{BF1,BFI,P} the
authors study spectral properties of sub-Laplace operators on nilpotent groups of step
two and on compact quotients by discrete subgroups. They determine the heat kernels
of these operators. This allows an explicit determination of the spectrum of the
sub-Laplacian, which is discrete in this situation.

The aim of this paper is to study spectra of sub-Riemannian analogs of the classical
Dirac operator. In the definition of the sub-Dirac operator the following difficulty
occurs: In contrast to the Riemannian case where we have the Levi-Civita connection
as a preferred connection, in general, there is no connection canonically associated
with a sub-Riemannian structure. Only in special geometric situations a canonical
connection exists. Hence the following definition of the sub-Dirac operator depends
on the choice of a connection. 

Let $(M,\cH,g)$ be a sub-Riemannian manifold, $\dim {\cal H}=d$. Suppose that $\nabla$
is a metric connection on $\cH$. Moreover, assume that $\cH$ is oriented and that
the bundle of orthonormal frames of $\cH$ admits a reduction to $\Spin(d)$. Such
a reduction will be called a spin structure of $\cH$. Then we can associate a spinor
bundle $S$ with this spin structure. Moreover, using the connection $\nabla$ we
can define a sub-Riemannian Dirac operator, which acts on sections in $S$.

Only few results for sub-Riemannian analogs of the Dirac operator are known. For the
case of a sub-Riemannian manifold of contact type such an operator was introduced
and studied by Petit \cite{Pe}, who called this operator Kohn-Dirac operator. More
exactly, this was done for a $\Spin^c$-structure.

Studying the sub-Riemannian Dirac operator the following natural questions arise:
Which structure does its spectrum have? How does the spectrum depend on the
sub-Riemannian geometry of the manifold and on the spin structure of $\cH$? How do
the sub-Dirac operator and its spectrum depend on the chosen connection?

In general, i.e., for arbitrary metric connections in $\cH$, the sub-Dirac is not symmetric.
We will characterize the symmetry of this operator by a simple condition on the connection.

Here we focus on nilmanifolds. More precisely, we study sub-Dirac operators on manifolds 
of the form  $M=\GG$ where $G=\RR^n\rtimes_A\RR$ is a semi-direct product defined by 
a one-parameter subgroup $A(t)$ of unipotent matrices in $\GL(n,\RR)$ and
$\Gamma$ is the subgroup $\ZZ^n\rtimes_A\ZZ$. These manifolds $M$ can be interpreted
as a suspension of the diffeomorphism of the torus $\RR^n/\ZZ^n$ induced
by $A(1)$. This is also the starting point of \cite{J} where the spectrum of the
Laplacian on left-invariant differential forms on $M$ is considered. Our sub-Dirac
operators will be associated with sub-Riemannian structures $(\dot{\cH},\dot{g})$ 
on $\GG$ coming from a left-invariant and bracket-generating distribution $(\cH,g)$
on $G$. We choose a metric connection in $H$ such that $D$ is symmetric. 

Our approach is to give an explicit decomposition of the regular representation of $G$.
Roughly speaking, it turns out that the sub-Dirac operator is an orthogonal sum of
elliptic operators on the real line, each having a discrete spectrum. This shows that
$D$ on $\GG$ has pure point spectrum.

We apply our results to compute the spectrum of $D$ explicitly for two classes of
two-step nilmanifolds of the above form. First we consider three-dimensional Heisenberg
manifolds. Secondly, we study a class of five-dimensional two-step nilpotent nilmanifolds
with a three-dimensional distribution. The latter example shows that the spectrum of the
sub-Dirac operator is not necessarily a discrete subset of $\RR$ and that its eigenvalues
may have infinite multiplicity, contrary to the results for the spectrum of the
sub-Laplacian on compact $2$-step nilmanifolds.

Finally, we discuss a three-step nilpotent example of dimension four with a
two-dimen\-sional distribution. In this case the spectrum can be expressed in terms
of the spectra of the family of operators $P_c=\del_t^2+(t^2+c)^2\pm 2t$, $c\in\RR$.

In all three examples, the multiplicities of the eigenvalues of $D$ can be read off
from the coadjoint orbit picture.
\\[2ex]
{\bf Acknowledgements} We would  like to thank Paul-Andi Nagy for several useful discussions.

\section{Sub-Riemannian Dirac operators}

\subsection{Definition of sub-Dirac operators}
Let $M$ be a smooth manifold and let $\cH\subset TM$ be a smooth distribution, where
$\dim\cH_x=d$ for all $x\in M$. Let $\Gamma(\cH)$ denote the space of smooth sections
of $\cH$. We assume that $\cH$ is bracket-generating. That means, that for each $x\in M$
there is a $J\in\NN$ such that the sequence
$$\Gamma_0:=\Gamma(\cH),\quad \Gamma_{j+1}:=\Gamma_j+[\Gamma_0,\Gamma_j]$$
satisfies $\{X(x)\mid X\in \Gamma_{J}\}=T_xM$. If $g$ is a Riemannian metric on $\cH$,
then the pair $(\cH,g)$ is called a sub-Riemannian structure on $M$ and $(M,\cH,g)$
is called a sub-Riemannian manifold.

Let $\nabla:\Gamma(\cH)\otimes\Gamma(\cH)\rightarrow\Gamma(\cH)$ be a metric connection on
$\cH$. Note that here we consider only derivations by vector fields in $\cH$. Suppose that $\cH$
is oriented and that it admits a spin structure, i.e., that there is a $\Spin(d)$-reduction
$P_{\Spin}(\cH)$ of the principal $\SO(d)$-bundle $P_{\SO}(\cH)$ of oriented orthonormal frames of
$(\cH,g)$. We consider  the complex representation of $\Spin(d)$ which is obtained by restriction of
(one of) the complex irreducible representation(s) of the Clifford algebra ${\cal C}l(d):=\Cl(\RR^d)$.
We will call it spinor representation and denote it by $\Delta_d$. The associated bundle
$P_{\Spin}(\cH)\times_{\Spin(d)}\Delta_d$ is called spinor bundle $S$ of $(\cH,g)$. The
space of smooth sections in $S$ is denoted by $\Gamma(S)$. The connection $\nabla$ defines
a connection  $\nabla^S:\Gamma(\cH)\times \Gamma(S)\rightarrow \Gamma(S)$ in the following
way. Let $s_1,\dots, s_d$ be a local orthonormal frame of $\cH$ and consider the local
connection forms $\omega_{ij}=g(\nabla s_i,s_j)$. Then we define
$$\nabla^S_X\ph:=X(\ph)+{\textstyle\frac12}\sum_{i<j} \omega_{ji}(X)\;s_i\cdot s_j\cdot\ph,$$
where `$\,\cdot\,$' denotes the Clifford multiplication.

Now we can define a sub-Riemannian Dirac operator, or sub-Dirac operator for short, by
\begin{equation}\label{ED} 
D=\sum_i s_i\cdot \nabla^S_{s_i}:\Gamma(S)\longrightarrow \Gamma(S),
\end{equation}
where again $s_1,\dots, s_d$ is a local orthonormal frame of $\cH$. Note, that the definition
of $D$ depends on the choice of the connection $\nabla$ on $\cH$ and that, in general, this
choice is far from being canonical in contrast to the Riemannian case, where we have the
Levi-Civita connection as a preferred connection.

A large class of metric connections in $\cH$ can be obtained in the following way. Suppose we
are given a further distribution ${\cal V}\subset TM$ such that $TM=\cH \oplus{\cal V}$. Then
this decomposition of $TM$ gives us a projection $\proj:TM\rightarrow \cH$ and we can define
a connection $\nabla$ by the Koszul formula
\begin{eqnarray}\label{EK}
2g(\nabla_X Y,Z)&=& X(g(Y,Z))+Y(g(X,Z))-Z(g(X,Y))\nonumber \\
&&+g(\proj[X,Y],Z)-g(\proj[X,Z],Y)-g(\proj [Y,Z],X),
\end{eqnarray}
where $X,Y,Z\in\Gamma(\cH)$. In this case $\nabla$ is uniquely determined by the vanishing
of $\nabla_XY-\nabla_YX-\proj[X,Y]$.

\subsection{Symmetry of the sub-Dirac operator}\label{S2.2}

A sub-Riemannian manifold is said to be regular if for each $j=1,\dots,J$ the dimension
of $\{X(x)\mid X\in \Gamma_{j}\}$ does not depend on the point $x\in M$. Let $(M,\cH,g)$
be an orientable regular sub-Riemannian manifold. Then $(M,\cH,g)$ admits an intrinsic
volume form $\omega_0$ on $M$, see \cite{M}, Section 10.5 and \cite{ABGR}. 

Let $(M,\cH,g)$ be an oriented and regular sub-Riemannian manifold. Consider any volume form
$\omega$ on $M$, e.g., the intrinsic one. Define the divergence of a vector field $X$ on $M$
by ${\cal L}_X\omega=(\Div X)\cdot \omega$. Let $\nabla$ be a metric connection on $\cH$.
Suppose that $\cH$ admits a spin structure and define $D:\Gamma(S)\rightarrow \Gamma(S)$
as above. Let $\ip$ be a hermitian inner product  on $\Delta_d$ for which the Clifford
multiplication is antisymmetric. This inner product is unique up to scale. It induces a
hermitian inner product on $S$, which together with $\omega$ gives an $L^2$-inner product
$\rip$ on the space $\Gamma_0(S)$ of sections in $S$ with compact support. 

It is easy to find examples of three-dimensional Heisenberg manifolds with two-dimensional
distribution $\cH$ and metric connection for which $D$ is not symmetric, see Section~\ref{S4.2}.

The following lemma states that the sub-Dirac operator is symmetric if and only if the
divergence defined by the sub-Riemannian structure coincides with the divergence given
by the connection, compare also \cite{FS} for the Riemannian case.

\begin{lm}\label{lm:sym_of_D}
Under the above conditions,  $D$ is symmetric if and only if 
\begin{equation}\label{Ediv}\Div X = \sum_{i=1}^d \la \nabla_{s_i}X,s_i\ra 
\end{equation}
holds for one (and therefore for every) local orthonormal basis $s_1,\dots, s_d$ of $\cH$.

If, in addition, ${\cal V}$ is a complement of $\cH$ in $TM$ and $\nabla$ is defined
as in {\rm (\ref{EK})}, then {\rm (\ref{Ediv})} is equivalent to the following condition. 
For one (and therefore for all) sets
$\{\xi_1,\ldots,\xi_{l}\}$, $l=\dim M-k$, of local sections of ${\cal V}$ that satisfy
$\omega(s_1,\ldots,s_d,\xi_1,\ldots,\xi_{l})=1$ the equation 
\[\eta_1([X,\xi_1])+\dots+\eta_l([X,\xi_l])=0\]
holds for all $X\in\Gamma(\cH)$, where 
$\eta_1,\ldots,\eta_{l}\in\Gamma(T^\ast M)$ are defined to be zero
on $\cH$ and dual to $\xi_1,\ldots,\xi_{l}$.

In particular, if $\codim\,\cH=1$, then $D$ is symmetric if and only if
$[\Gamma(\cH),\xi_1]\subset\Gamma(\cH)$.
\end{lm}

\proof Consider sections $\ph,\psi\in\Gamma_0(S)$ and define
$f:\cH\rightarrow \CC$ by 
\begin{equation}\label{E1a}f(w):= \langle \ph,w \cdot \psi\rangle.
\end{equation}
Moreover, define
$u\in \Gamma_0(\cH\otimes \CC)$ by 
\begin{equation}\label{E1} g^{\Bbb C}(u,w)=f(w)
\end{equation} 
for all $w\in\Gamma(\cH)$, where $g^{\Bbb C}$ denotes the the complex bilinear extension of $g$.
Choose a local orthonormal frame $s_1,\dots, s_d$ of $\cH$. Then 
$$\langle D\ph,\psi \ra-\la\ph,D\psi\ra =\sum_{i=1}^d\Big( f(\nabla_{s_i}s_i)-s_i(f(s_i))\Big)
=\sum_{i=1}^d g^{\Bbb C}(\nabla_{s_i}u,s_i),$$
thus
$$(D\ph,\psi)-(\ph,D\psi)=\int_M \left(\sum_{i=1}^d g^{\Bbb C}(\nabla_{s_i}u,s_i)\right)\omega
=\int_M \left(\sum_{i=1}^d g^{\Bbb C}(\nabla_{s_i}u,s_i)-\Div(u)\right)\omega.$$  
In particular, (\ref{Ediv}) is sufficient for the symmetry of $D$. On the other hand, any
section $u_1\in\Gamma_0(\cH)$ is the real part of a section $u\in \Gamma_0(\cH\otimes \CC)$
that satisfies~(\ref{E1a}) and (\ref{E1}) for some $\ph,\psi\in\Gamma_0(S)$. Indeed,
choose $\psi$ such that $\la \psi(x),\psi(x)\ra=1$ for all $x\in {\rm supp}\,u_1$ and
put $\ph:= u_1\cdot\psi$. Define $u$ by (\ref{E1a}) and (\ref{E1}). Then
$$g(u_1,w)=\Re \la u_1\cdot\psi,w\cdot\psi\ra=\Re \la \ph,w\cdot\psi\ra=\Re f(w)$$
for all $w\in\Gamma(\cH)$, hence $u_1=\Re u$. Consequently, the symmetry of $D$ implies
$$\int_M \left(\sum_{i=1}^d g(\nabla_{s_i}u,s_i)-\Div(u)\right)\omega=0$$ for all
$u\in\Gamma_0(\cH)$. Since the integrand is $C^\infty_0(M)$-linear in $u$,
Equation (\ref{Ediv}) follows.

The second part of the lemma now follows from
\begin{eqnarray*}
&&\;\Div(u)\ =\ ({\cal L}_u\omega)(s_1,\dots,s_d,\xi_1,\dots,\xi_l)\\
&=& -\sum_{i=1}^d\omega(s_1,\ldots,[u,s_i],\dots, s_d,\xi_1,\dots,\xi_l)
-\sum_{j=1}^l\omega(s_1,\dots,s_d,\xi_1,\ldots,[u,\xi_j],\dots,\xi_l)\\
&=& -\sum_{i=1}^d g(\proj [u,s_i],s_i)-\sum_{j=1}^l \eta_j([u,\xi_j])
\ =\ \sum_{i=1}^d g(\nabla_{s_i}u,s_i)-\sum_{j=1}^l \eta_j([u,\xi_j]),\\
\end{eqnarray*}
where the last equality is a consequence of Equation~(\ref{EK}).\qed

\subsection{Sub-Dirac operators on Lie groups and compact quotients}

Let $G$ be a simply connected Lie group and $\Gamma\subset G$ a uniform discrete subgroup.
Let $\cH\subset TG$ be a left-invariant distribution and $ g$ a left-invariant
Riemannian metric on $\cH$. Obviously, $\cH$ is spanned by orthonormal
left-invariant vector fields $s_1,\dots, s_d$. In particular, the frame bundle
$P_{\SO}(\cH)$ is a trivial bundle and the unique spin structure of $\cH$
equals $P_{\Spin}(\cH)=G\times \Spin(d)$. 

The pair $(\cH, g)$ induces a sub-Riemannian structure on $\Gamma\setminus G$,
which we will  denote by $(\dot\cH,\dot g)$. The frame bundle $P_{\SO}(\dot\cH)$
can be identified with
$$P_{\SO}(\dot\cH)=G\times_\Gamma \SO(d),$$ 
where $\Gamma$ acts by left multiplication on $G$ and trivially on $\SO(d)$. There is
a one-to-one correspondence between homomorphisms $\eps:\Gamma\rightarrow\ZZ_2=\{0,1\}$
and spin structures of $\dot\cH$ given by
$$\eps\longmapsto P_{\Spin,\eps}(\dot\cH)=G\times_\Gamma \Spin(d),$$
where $\gamma\in \Gamma$ acts by multiplication by $e^{i\pi\eps(\gamma)}$
on $\Spin(d)$. Spinor fields are sections of the associated spinor bundle
$P_{\Spin,\eps}(\dot\cH)\times_{\Spin(d)}\Delta_d\cong{G\times_\Gamma}\Delta_d$ or,
equivalently, maps $\psi:G\rightarrow \Delta_d$ that satisfy
$\psi(\gamma g)=e^{i\pi\eps(\gamma)}\psi(g)$ for all $\gamma\in\Gamma$, $g\in G$.

The intrinsic volume form $\omega_0$ on $\Gamma\setminus G$ introduced in \cite{ABGR}
and discussed in Section~\ref{S2.2} is left-invariant. 

Now let $\nabla$ be a left-invariant metric connection on $\dot\cH$. Let $s_1,\dots, s_d$
be an orthonormal basis of $\dot\cH$ consisting of left-invariant vector fields. As for
the symmetry of the Dirac operator discussed in Lemma~\ref{lm:sym_of_D}, note that
Equation~(\ref{Ediv}) is equivalent to
\begin{equation}
0=\sum_{i=1}^d\,g(\nabla_{s_i}s_j,s_i)
\end{equation}
for all $j=1,\dots,d$. Indeed, since the intrinsic volume form is left-invariant and $G$
must be unimodular the divergence of any left-invariant vector field vanishes.

\section{Sub-Riemannian structures on $\Gamma\setminus (\RRu ^n\rtimes_A\RRu)$}

\subsection{The standard model}\label{S3.1}
Let $A(t)=\exp(tB)$ be a one-parameter subgroup of $\GL(n,\RR)$.  We consider the
simply-connected solvable Lie group $\RR^n\rtimes_A\RR$ with group law
\[(x,s)\,(y,t)=(x+A(s)y\,,\,s+t)\;.\]
In particular, $(0,t)\,(x,0)\,(0,t)^{-1}=(A(t)x\,,\,0)$. In addition, we
assume $A(1)\in\SL(n,\ZZ)$ so that the set $\ZZ^n\times\ZZ$ becomes a uniform
discrete subgroup.

The pair $(\RR^n\rtimes_A\RR,\ZZ^n\rtimes_A\ZZ)$ serves as a standard model in the
following sense.
\begin{lm}\label{lm:model}
Let $G$ be an exponential Lie group admitting a connected abelian normal subgroup $N$
of codimension one. Let $\Gamma$ be a uniform discrete subgroup of $G$ such that
$\Gamma\cap N$ is uniform in $N$. Then there exists a one-parameter subgroup $A$
of~$\GL(n,\RR)$, $n=\dim N$, with $A(1)\in\SL(n,\ZZ)$, and an isomorphism~$\Phi$
of $\RR^n\rtimes_A\RR$ onto $G$ mapping $\ZZ^n\rtimes_A\ZZ$ onto $\Gamma$.
\end{lm}
\begin{proof}
We fix generators $v_1,\ldots,v_n$ of the lattice $\Gamma\cap N$ of the vector group $N$
and consider the linear isomorphism $M$ of $\RR^n$ onto $N$ given by $M(e_j)=v_j$. On
the other hand, the assumption on $\Gamma$ implies that $\Gamma N$ is closed in $G$
and that $\Gamma N/N$ is a discrete subgroup of $G/N$. Hence there exists $b\in\fg$
with $\exp(b)\in\Gamma$ and such that $\exp(b)N$ is a generator of $\Gamma N/N$.

Put $A(t)x=M^{-1}(\exp(tb)M(x)\exp(tb)^{-1})$. Now it follows that $\Phi(x,t)=M(x)\exp(tb)$
is an isomorphism of $\RR^n\rtimes_A\RR$ onto $G$ with $\Phi(\ZZ^n\times\ZZ)=\Gamma$.
In particular, $\ZZ^n\rtimes_A\ZZ$ is a subgroup of $\RR^n\rtimes_A\RR$. This means
that $\ZZ^n$ is $A(l)$-invariant for all $l\in\ZZ$ so that $A(l)\in\SL(n,\ZZ)$. \qed
\end{proof}

The condition $A(1)\in\SL(n,\ZZ)$ implies $B\in\fsl(n,\RR)$ and $A(t)\in\SL(n,\RR)$ for
all $t\in\RR$. This reflects the fact that locally compact groups admitting a uniform
discrete subgroup are unimodular, compare Theorem~7.1.7 of~\cite{W}.

The Lie algebra of $G:=\RR^n\rtimes_A\RR$ is isomorphic to $\fg=\RR^n\rtimes_B\RR$, and
$B=\ad(b)\,|\,\fn$, where $b=(0,1)$ and $\fn=\RR^n\times\{0\}$. Note that $G$ is
exponential if and only if $B$ has no purely imaginary eigenvalues, compare Theorem~1
of~\cite{LL}.

It is evident that
\[\pi(x,t)=\left( \begin{array}{cccc} & & & x_1 \\ & A(t) & & \vdots \\
& &  & x_n \\ 0 & \ldots & 0 & 1 \end{array}\right)\]
defines a representation, which is faithful provided that $G$ is exponential and
not abelian.

\begin{ex}\label{ex:3_dim_Heis}
{\rm Fix $r\in\ZZ_+$ and set
$B=\left(\begin{array}{cc} 0 & r\\ 0 & 0\end{array}\right)$ so that
$A(t)=\exp(tB)=I+tB$. Since $A(l)\in\SL(2,\ZZ)$ for $l\in\ZZ$,
$\Gamma=\ZZ^2\rtimes_A\ZZ$ is a subgroup of $G=\RR^2\rtimes_A\RR$. On the other hand,
\[\pi(x_1,x_2,t)=\left(\begin{array}{ccc} 1 & rt & x_1 \\ 0 & 1 & x_2 \\
 0 & 0 & 1\end{array}\right)\]
gives an isomorphism from $G$ onto the three-dimensional Heisenberg group $H(1)$
in its standard realisation as a group of matrices mapping $\Gamma$ onto
\[\Gamma_r=\left\{\;\left(\begin{array}{ccc} 1 & rl & k_1 \\ 0 & 1 & k_2 \\ 0 & 0 & 1
\end{array}\right):l,k_1,k_2\in\ZZ\;\right\}\;.\]
In particular, the above construction yields all uniform discrete subgroups of $H(1)$
and hence all three-dimensional Heisenberg manifolds, compare Section~2 of~\cite{GW}.}
\end{ex}

Heisenberg manifolds and certain generalisations of them will be discussed in
Section~\ref{S4.2} and~\ref{S4.3} in greater detail.

For the subgroup $\Gamma:=\ZZ^n\rtimes_A\ZZ$, the spin structures of any distribution
of $T(\GG)$ induced by a left-invariant distribution of $T(G)$ are determined as follows.

\begin{lm}\label{lm:eps} A map $\eps:\Gamma\to\ZZ_2$ is a homomorphism if
and only if $\eps(k,l)=\eps'(k)+\dot\eps(l)$ for some homomorphism
$\dot{\eps}:\ZZ\to\ZZ_2$ and a homomorphism $\eps':\ZZ^n\to\ZZ_2$ satisfying
\begin{equation}\label{EsumA}
\sum_\mu \eps'(e_\mu)(A(1)-I)_{\mu\nu}\,\in 2\ZZ\end{equation}
for all $\nu$.
\end{lm}
\begin{proof}
Any homomorphisms $\eps:\Gamma\to\ZZ_2$ defines homomorphisms
$\dot{\eps}:\ZZ\to\ZZ_2$, $\dot{\eps}(l)=\eps(0,l)$, and
$\eps':\ZZ^n\to\ZZ_2$, $\eps'(k)=\eps(k,0)$, where $\eps'$ satisfies
\[\eps'(A(l)k)=\eps(A(l)k,0)=\eps(\,(0,l)(k,0)(0,l)^{-1}\,)=\eps(k,0)=\eps'(k)\]
for all $l\in\ZZ$ and $k\in\ZZ^n$. The latter condition reduces to
$\eps'(A(1)k)=\eps'(k)$ for all $k$, and hence to
$\eps'(e_\nu)=\eps'(A(1)e_\nu)=\sum_\mu A_{\mu\nu}(1)\,\eps'(e_\mu)$
in $\ZZ_2$ for all $\nu$, which is equivalent to (\ref{EsumA}). It is easy
to check that the converse holds also true.\qed
\end{proof}

Since $\Gamma$ is uniform and discrete, there exists a unique normalised
right $G$-invariant Radon measure $\mu$ on~$\GG$ which can be obtained
as follows: Let $I_n=[0,1]^n$ denote the unit cube of $\RR^n$. Then
\begin{equation}\label{eq:inv_measure}
\int_{\GG}\ph\;d\mu=\int_0^1\Big(\int_{I_n}\;\ph(x,s)\;dx\Big)\;ds
\end{equation}
for all $\ph\in C(\GG)$. This can be proved using that the Haar measure of~$G$ equals
the Lebesgue measure of~$\RR^{n+1}$ and that $F=[0,1)^{n+1}$ is a fundamental set
for~$\Gamma$ on~$G$.

\subsection{Decomposition of the right-regular representation}

Let $A(t)$ be one-parameter subgroup of $\GL(n,\RR)$ such that $G=\RR^n\rtimes_A\RR$
is exponential and $\Gamma=\ZZ^n\rtimes_A\ZZ$ is a subgroup of $G$. Let
$\eps:\Gamma\to\ZZ_2$ be a group homomorphism. Our aim is to decompose the
right regular representation of $G$ on $L^2(G,\eps)$.

Let $C(G,\eps)$ denote the space of all continuous $\CC$-valued functions $\ph$
on $G$ satisfying $\ph(gy)=e^{i\pi\eps(g)}\,\ph(y)$ for all $g\in\Gamma$ and $y\in G$,
and $L^2(G,\eps)$ the completion of $C(G,\eps)$ with respect to the Hilbert norm
\[|\,\ph\,|_{L_2}^2=\int_{\GG}|\,\ph\,|^2\;d\mu\;.\]
Now right translation $\rho(x)\ph\;(y)=\ph(yx)$ gives rise to a unitary representation
of $G$ on~$L^2(G,\eps)$. This is precisely the definition of the induced representation
$\rho=\ind_\Gamma^G e^{i\pi\eps}$ of the unitary character $e^{i\pi\eps}$ of~$\Gamma$.
By Theorem~7.2.5 of \cite{W}, $\rho$ can be written as a countable orthogonal sum of
irreducible subrepresentations with finite multiplicity.
We will give such a decomposition explicitly, generalizing the results of $\cite{AB}$
for the three-dimensional Heisenberg group, which motivated this article.

To this end, we consider partial Fourier transformation with respect to the first n
variables: If $\ph\in C(G,\eps)$, then $\ph(k+A(l)x,l+t)=e^{i\pi\eps(k,l)}\,\ph(x,t)$
for all $(k,l)\in\Gamma$. In particular, $\ph(2k+x,t)=e^{i\pi\eps(2k,0)}\,\ph(x,t)=\ph(x,t)$
which shows that $x\mapsto\ph(x,t)$ is $2\ZZ^n$-invariant. For such functions
it is natural to consider 
\[\widehat\ph(\xi,t)=\int_{I_n} \ph(2x,t)\,e^{-2\pi i\langle\xi,x\ra}\,dx\]
for $\xi\in\ZZ^n$. Clearly $\ph$ is uniquely determined by its Fourier coefficient
functions.

It follows from~(\ref{eq:inv_measure}) that the restriction of $\ph\in L^2(G,\eps)$
to $2I_n\times[0,1]$ is $L^2$- and hence $L^1$-integrable with respect to the
Lebesgue measure. In particular, the integral in the definition of
$(F_\xi\cdot\ph)(t):=\widehat\ph(\xi,t)$ makes sense for almost all $t$.

\begin{pr}\label{pr:Plancherel}
For $\ph\in L^2(G,\eps)$ there holds the Plancherel formula
$$|\,\ph\,|^2_{L^2}=\sum_{\xi\in{\Bbb Z}^n}\;\int_0^1|\widehat \ph(\xi,t)|^2\;dt\;.$$
\end{pr}
\begin{proof}
By the Plancherel theorem for $L^2$-functions on the torus, we obtain
$$|\,\ph\,|^2_{L^2}=\int_0^1\Big(\int_{I_n}\;|\,\ph(x,t)\,|^2\;dx\Big)\;dt
=\int_0^1\sum_{\xi \in{\Bbb Z}^n}\;|\widehat \ph(\xi,t)|^2\;dt\;,$$
where summation and integration can be interchanged. \qed
\end{proof}

As a corollary we get $F_\xi\cdot\ph\in L^2(\,[0,1]\,)$. Moreover, the
inequality $|\,F_\xi\cdot\ph\,|_{L^2}\le |\,\ph\,|_{L_2}$ shows that
$F_\xi:L^2(G,\eps)\to L^2(\,[0,1]\,)$ is a continuous linear operator.

Taking into account
\[\int_{I_n}g(Mx)\;dx=\int_{I_n}g(x)\;dx\]
for integrable $\ZZ^n$-invariant functions $g$ on $\RR^n$ and $M\in\SL(n,\ZZ)$,
one can prove that the $\eps$-equivariance of $\ph$ entails the following
conditions on its Fourier transform.

\begin{lm}\label{lm:eps_equiv}
For $\ph\in L^2(G,\eps)$ and $(k,l)\in\Gamma$ it holds
\begin{eqnarray*}
e^{i\pi\eps(k,l)}\;\widehat\ph(\xi,t)=e^{\pi i\la\,A(-l)^\top\xi\,,\,k\,\ra}\;
\widehat\ph(\,A(-l)^\top\xi\,,\,l+t\,)\;,
\end{eqnarray*}
where $A(l)^\top$ denotes the transpose of the operator $A(l)$ with respect
to the standard inner product on $\RR^n$.
\end{lm}

The one-parameter subgroup $A$ represents the adjoint representation of the
subgroup $\RR\cong\{0\}\times\RR$ of $G$ on the Lie algebra $\RR^n$ of the
normal subgroup $\RR^n\times\{0\}$. Identifying the linear dual of
this Lie algebra with $\RR^n$ by means of the standard inner product, we
see that $t\mapsto A(t)^\top$ is the coadjoint representation. The preceding
lemma reveals the importance of this group action in the present context.
Any $\xi\in\RR^n$ has a $\ZZ$-orbit $\theta=\{A(l)^\top\xi:l\in\ZZ\}$ and an
$\RR$-orbit $\omega=\{A(t)^\top\xi:t\in\RR\}$.

Let $\ph\in L^2(G,\eps)$ and $\xi\in\ZZ^n$. The equality 
$\widehat{\ph}(\,A(l)^\top\xi,t)=e^{i\pi\eps(0,-l)}\,\widehat{\ph}(\xi,l+t)$ shows
that $\widehat{\ph}(\xi,\cdot)$ determines $\widehat{\ph}(\eta,\cdot)$ for
all $\eta\in\theta$. In particular, $\supp\,\widehat{\ph}:=\{\xi\in\ZZ^n:
\widehat{\ph}(\xi,\cdot)\neq 0\}$ is a $\ZZ$-invariant subset.

\begin{lm}\label{lm:Sigma}
The set
\[\Sigma_{\eps'}:=\{\,\xi\in\ZZ^n:\xi_\nu\in 2\ZZ+\eps'(e_\nu)
\mbox{ \textup{for all} }1\le\nu\le n\,\}\]
is $\ZZ$-invariant and contains $\supp\,\widehat{\ph}$ for every $\ph\in L^2(G,\eps)$.
\end{lm}
\begin{proof}
Let $\xi\in\Sigma_{\eps'}$ and $l\in\ZZ$. Since $\eps'(A(l)^\top e_\nu)=\eps'(e_\nu)$
and $\la\xi,k\ra\in 2\ZZ+\eps'(k)$ for all $k\in\ZZ^n$, it follows
$\la A(l)^\top\xi,e_\nu\ra=\la\xi,A(l)e_\nu\ra\in 2\ZZ+\eps'(e_\nu)$ and
$A(l)^\top\xi\in\Sigma_{\eps'}$. This proves $\Sigma_{\eps'}$ to be
$\ZZ$-invariant. Let $\ph\in L^2(G,\eps)$ and $\xi\not\in\Sigma_{\eps'}$. Then
there is $1\le\nu\le n$ such that $\xi_\nu\not\in 2\ZZ+\eps'(e_\nu)$ and
hence $ e^{\pi i\eps'(e_\nu)}\neq e^{\pi i\la\xi,e_\nu\ra}$. On the other hand,
by Lemma~\ref{lm:eps_equiv} we have $e^{i\pi\eps(k,0)}\,\widehat{\ph}(\xi,t)
=e^{\pi i\la\xi,k\ra}\,\widehat{\ph}(\xi,t)$ for all $k$. This implies
$\widehat{\ph}(\xi,\cdot)=0$. \qed
\end{proof}

Note that $\xi\in\Sigma_{\eps'}$ implies $e^{i\pi\eps(k,0)}=e^{i\pi\la\xi,k\ra}$
for all $k\in\ZZ^n$.

Let $\ZZ\setminus\Sigma_{\eps'}$ denote the set of all $\ZZ$-orbits in
$\Sigma_{\eps'}$. For $\theta\in\ZZ\setminus\Sigma_{\eps'}$, it follows that
\[U_\theta\,:=\,\bigcap_{\xi\not\in\theta}\ker F_\xi\,=\,\{\ph\in L^2(G,\eps):
\supp\,\widehat{\ph}\subset\theta\}\,\neq\,0\]
is a closed subspace of $L^2(G,\eps)$. It will be shown below that $U_\theta$
is $\rho(G)$-invariant.

\begin{pr}\label{pr:L2_decomp}
These subspaces form an orthogonal decomposition
\[L^2(G,\eps)=\mathop{\overline{\bigoplus}}_{\theta\in{\Bbb Z}\setminus\Sigma_{\eps'}}
U_\theta.\]
\end{pr}
\begin{proof}
Let $\theta_1,\theta_2\in\ZZ\setminus\Sigma_{\eps'}$ be distinct orbits. By
Proposition~\ref{pr:Plancherel} and the polarisation identity, we obtain
\[\la\,\ph_1,\ph_2\,\ra_{L^2}\;=\;\sum_{\xi\in\Bbb Z^n}\,\int_0^1\,
\widehat\ph_1(\xi,t)\,\overline{\widehat\ph_2(\xi,t)}\;dt\;=\;0\]
for $\ph_1\in U_{\theta_1}$ and $\ph_2\in U_{\theta_2}$. Hence $U_{\theta_1}$
and $U_{\theta_2}$ are orthogonal. It remains to be shown that the direct sum of
the $U_\theta$ is dense. Since $C^\infty(G,\eps)$ is dense in $L^2(G,\eps)$,
it suffices to prove that every smooth $\eps$-equivariant function can
be approximated by a finite sum of functions in the $U_\theta$. By the decay
of the Fourier transform of $\ph\in C^\infty(G,\eps)$, it follows that
\[\ph_\theta(x,t)=\sum_{\xi\in\theta}\;\widehat\ph(\xi,t)\,e^{\pi i\la \xi,x\ra}\]
is a smooth function in $U_\theta$ and that
$\ph=\sum_{\theta\in{\Bbb Z}\setminus\Sigma_{\eps'}}\,\ph_\theta$ converges uniformly
on~$\RR^n\times [0,1]$. In particular, $|\,\ph-\sum_{\theta\in J}\ph_\theta\,|_{L^2}\to 0$
for $J\subset\ZZ\setminus\Sigma_{\eps'}$ finite and increasing. (This also proves
that $U_\theta\cap C^\infty(G,\eps)$ is dense in $U_\theta$.) \qed
\end{proof}

The right regular representation $\rho(x,s)\ph\;(y,t)=\ph(y+A(t)x,t+s)$ on
$L^2(G,\eps)$ is compatible with partial Fourier transform in the sense that
\begin{equation}\label{eq:rho_xi}
(\rho(x,s)\ph)\,^{\widehat{}}\ (\xi,t)=e^{\pi i\la A(t)^\top\xi,x\ra}\;
\widehat\ph(\xi,t+s)\;.
\end{equation}
In particular, $\widehat\ph(\xi,\cdot)=0$ implies
$(\rho(x,s)\ph)\,^{\widehat{}}\ (\xi,\cdot)=0$  for all $(x,s)\in G$
which proves $U_\theta$ to be $\rho(G)$-invariant. We define
$\rho_\theta=\rho\,|\,U_\theta$.

For any $\RR$-orbit $\omega\subset\RR^n$, we consider the stabilizer
$\dot{H}_\omega=\{t\in\RR:A(t)^\top\xi=\xi\}$ whose definition does not
depend on the choice of the point $\xi\in\omega$. Since $t\mapsto A(t)^\top$
is the coadjoint representation of an exponential Lie group, we know that the closed
subgroup $\dot{H}_\omega$ is connected, see p.\ 49 of~\cite{LL}. Thus there are
only two possibilities, either $\dot{H}_\omega=\RR$ or $\dot{H}_{\omega}=\{0\}$.

Suppose that $\theta\subset\Sigma_{\eps'}$ is a $\ZZ$-orbit such that
$\dot{H}_\omega=\RR$ where $\omega$ is the unique $\RR$-orbit containing~$\theta$.
Then $\omega=\theta=\{\xi\}$ is a fixed point. We claim that
$(T_\xi\cdot\ph)(t)=\widehat{\ph}(\xi,t)$ gives a unitary isomorphism
of $U_{\theta}$ onto $L^2(\RR,\dot{\eps})$. First
\[\widehat{\ph}(\xi,l+t)=\widehat{\ph}(\,A(-l)^\top\xi,l+t\,)
=e^{i\pi\dot{\eps}(l)}\,\widehat{\ph}(\xi,t)\]
by Lemma~\ref{lm:eps_equiv} which shows $T_\xi\cdot\ph\in L^2(\RR,\dot{\eps})$
for $\ph\in U_{\theta}$. Clearly $T_\xi$ is linear and
\[|\,T_\xi\cdot\ph\,|_{L^2}^2=\int_0^1|\,\widehat{\ph}(\xi,t)\,|^2\;dt
=|\,\ph\,|_{L^2}^2\]
by Proposition~\ref{pr:Plancherel}. If $\psi\in C(\RR,\dot{\eps})$,
then $\ph(x,t)=\psi(t)\,e^{\pi i\la\xi,x\ra}$ is in $C(G,\eps)$ by
Lemma~\ref{lm:Sigma}, and $T_\xi\cdot\ph=\psi$. Thus $T_\xi$ is onto.
From Equation~(\ref{eq:rho_xi}) it follows that the representation
$\rho_\xi(x,s)=T_\xi\,\rho_\theta(x,s)\,T^\ast_\xi$ on $L^2(\RR,\dot{\eps})$
is given by
\[\rho_\xi(x,s)\psi\,(t)=e^{\pi i\la \xi,x\ra}\;\psi(t+s)\;.\]
Define $\eps^\sharp(t)=e^{\pi i\dot{\eps}(1)t}$ for $t\in\RR$. Then
$\eps^\sharp$ is a unitary character of $\RR$ extending $\dot{\eps}$. Any
extension has the form $t\mapsto\eps^{\sharp}(t)\,e^{2\pi imt}$ for
some $m\in\ZZ$. By means of the unitary isomorphism
$(U\cdot\psi)(t)=\eps^\sharp(-t)\,\psi(t)$
of $L^2(\RR,\eps)$ onto $L^2(\TT)$, we see that $\rho_\xi$ is unitarily
equivalent to
\begin{equation}\label{eq:rho_xi_for_fixed_points}
\tilde{\rho}_\xi(x,s)\psi\;(t)=e^{\pi i\la\xi,x\ra}\;\eps^\sharp(s)\;
\psi(t+s)
\end{equation}
on $L^2(\TT)$. For $m\in\ZZ$ we consider the unitary character of $G$ given by
\[\chi_{\eps^{\sharp},\omega,m}(x,s)=e^{\pi i\la\xi,x\ra}\;e^{2\pi ims}\;
\eps^\sharp(s).\]
Finally, using the Fourier transformation and the Plancherel theorem
for $L^2$-functions on the torus, we conclude that $\rho_\theta$ is
unitarily equivalent to an orthogonal sum
\[\rho_\theta\;\cong\;\bigoplus_{m\in{\Bbb Z}}\chi_{\eps^{\sharp},\omega,m}\]
of $1$-dimensional subrepresentations. Note that up to isomorphism this
decomposition does not depend on the choice of the extension $\eps^\sharp$.

Now we suppose that $\theta\subset\Sigma_{\eps'}$ is a $\ZZ$-orbit such
that $\dot{H}_\omega=\{0\}$. This means that $\theta$ is an infinite set
and that the unique $\RR$-orbit $\omega$ containing $\theta$ is not
relatively compact. In this case we have
\begin{lm}\label{lm:rho_eta}
For every $\eta\in\omega$ there exists a unitary isomorphism $T_\eta$
of~$U_\theta$ onto~$L^2(\RR)$ which intertwines $\rho_\theta$ and
\begin{equation}\label{eq:rho_eta}
\rho_\eta(x,s)\psi\;(t)=e^{\pi i\la A(t)^\top\eta,x\ra}\;\psi(t+s)\;.
\end{equation}
\end{lm}
\begin{proof}
Let $\xi\in\theta$ and $r\in\RR$ such that $\eta=A(r)^\top\xi$. We claim that
$(T_\eta\ph)(t)=\widehat\ph(\xi,t+r)$ is a unitary isomorphism satisfying our
needs: First of all, Proposition~\ref{pr:Plancherel} implies
\begin{eqnarray*}
|\,T_\eta\ph\,|_{L^2}^2 &=& \int_{-\infty}^{+\infty}\,|\,\widehat{\ph}(\xi,t+r)\,|^2\;dt
=\int_{-\infty}^{+\infty}\,|\,\widehat{\ph}(\xi,t)\,|^2\;dt\\
&=& \sum_{l\in\Bbb Z}\;\int_0^1\,|\widehat{\ph}(\xi,l+t)\,|^2\;dt
= \sum_{l\in\Bbb Z}\;\int_0^1\,|\widehat{\ph}(A(l)^\top\xi,t)\,|^2\;dt=|\,\ph\,|_{L^2}^2
\end{eqnarray*}
which shows that $T_\eta\ph\in L^2(\RR)$ is well-defined and that $T_\eta$ is isometric.
If $\psi\in C_0(\RR)$, then the sum
\[\ph(x,t)=\sum_{l\in\Bbb Z}e^{-i\pi\dot{\eps}(l)}\;\psi(l+t-r)\;
e^{\pi i\la A(l)^\top\xi,x\ra}\]
is locally finite in $t$ and $\ph\in C(G,\eps)$ by Lemma~\ref{lm:Sigma}. Using
$A(l)^\top\xi\neq\xi$ for $l\neq 0$, we conclude $T_\eta\ph=\psi$. This proves
$T_\eta$ to be surjective. Finally, we observe that (\ref{eq:rho_eta}) is a
consequence of (\ref{eq:rho_xi}). \qed
\end{proof}

Let $X_{\eps'}^\infty$ denote the set of all $\RR$-orbits which intersect the subset
$\Sigma_{\eps'}$ of $\ZZ^n$ and which are not relatively compact. Let $X_{\eps'}^0$
be the set of all orbits of the form $\omega=\{\,\xi\,\}$ for a fixed point
$\xi\in\Sigma_{\eps'}$. If $\omega\in X_{\eps'}^\infty$, $\eta\in\omega$ is arbitrary,
and $\theta_1,\theta_2$ are $\ZZ$-orbits contained in $\omega\cap\Sigma_{\eps'}$,
then $\rho_{\theta_1}\cong\rho_\eta\cong\rho_{\theta_2}$ by Lemma~\ref{lm:rho_eta}.
This implies
\[\bigoplus_{\theta\in{\Bbb Z}\setminus(\omega\cap\Sigma_{\eps'})}\rho_\theta\;
\cong\;m_{\eps,\omega}\,\rho_{\omega}\]
where $m_{\eps,\omega}$ is the number of $\ZZ$-orbits contained in
$\omega\cap\Sigma_{\eps'}$, which is apriori known to be finite, and $\rho_\omega$
is the common unitary equivalence class of the representations $\rho_\theta$
for $\theta\in\ZZ\setminus(\omega\cap\Sigma_{\eps'})$.

Summing up the preceding conclusions, we obtain

\begin{theo}\label{theo:decomp}
Let $A(t)$ be one-parameter group of $\GL(n,\RR)$ with $A(1)\in\SL(n,\ZZ)$ and such
that $G=\RR^n\rtimes_A\RR$ is exponential. Let $\eps:\Gamma\to\ZZ_2$ be a homomorphism.
Then the right regular representation $\rho$ of $G$ in~$L^2(G,\epsilon)$
decomposes as follows:
\[\rho\;\cong\;\bigoplus_{\theta\in{\Bbb Z}\setminus\Sigma_{\eps'}}\rho_\theta\;
\cong\;\left(\,\bigoplus_{\omega\in X_{\eps'}^0}\bigoplus_{m\in{\Bbb Z}}
\chi_{\dot{\eps},\omega,m}\,\right)\;\bigoplus\;\left(\,
\bigoplus_{\omega\in X_{\eps'}^\infty}m_{\eps,\omega}\;\rho_\omega\,\right)\]
where the multiplicities $m_{\epsilon,\omega}=\#\ZZ\setminus(\omega\cap\Sigma_{\epsilon'})$
are finite, the
\[\chi_{\dot{\eps},\omega,m}(x,s)=e^{\pi i\la\xi,x\ra}\;e^{\pi i(\,2m+\dot{\eps}(1)\,)s}\]
are characters of $G$, and the $\rho_\omega$ are irreducible on $L^2(\RR)$. For
every $\eta\in\omega$,
\[\rho_\eta(x,s)\psi\;(t)=e^{\pi i\la A(t)^\top\eta,x\ra}\;\psi(t+s)\]
is a representative for the unitary equivalence class of $\rho_\omega$.
Moreover, the representations $\{\,\chi_{\dot{\eps},\omega,m}:\omega\in X_{\eps'}^0,
m\in\ZZ\} \cup\{\,\rho_\omega:\omega\in X_{\eps'}^\infty\,\}$ are mutually
inequivalent.
\end{theo}
\begin{proof}
It remains to verify the last assertion and the irreducibility of $\rho_\omega$.
Clearly characters are unitarily equivalent if and only if they are equal, and
not unitarily equivalent to a representation on~$L^2(\RR)$. Let $C^\ast(N)$ be
the enveloping $C^\ast$-algebra of the group algebra $L^1(N)$ of $N=\RR^n\times\{0\}$.
Recall that $C^\ast(N)$ is isomorphic to $C_\infty(\widehat{N})$ via Fourier
transformation. The above formula for $\rho_\eta$ shows that the $C^\ast$-kernel
of the integrated form of $\rho_\omega\,|\,N$ consists of all $g\in C^\ast(N)$
whose Fourier transform vanishes on $\omega$. Since the $\RR$-orbits are locally
closed, it follows that $\rho_{\omega_1}$ and $\rho_{\omega_2}$ are inequivalent
whenever $\omega_1\neq\omega_2$. -- If $U$ is a closed $\rho_\eta(G)$-invariant
subspace of $L^2(\RR)$, then $U$ is invariant under translations and multiplication
by bounded continuous functions. Thus it follows $U=\{0\}$ or $U=L^2(\RR)$. \qed
\end{proof}

\begin{lm}\label{lm:range_of_T_eta}
Suppose that $G=\RR^n\rtimes_A\RR$ is a nilpotent Lie group. Let
$\omega\in X_{\eps'}^\infty$ and $\theta\in\ZZ\setminus(\omega\cap\Sigma_{\eps'})$.
For $\eta\in\omega$ let $T_\eta$ denote the unitary isomorphism of $U_\theta$
onto $L^2(\RR)$ defined in the proof of Lemma~\ref{lm:rho_eta}. Then for every
Schwartz function $\psi\in\cS(\RR)$ there exists a (unique)
$\ph\in U_\theta\cap C^\infty(G,\eps)$ such that $T_\eta\ph=\psi$.
\end{lm}
\begin{proof}
Given $\psi\in\cS(\RR)$ we consider
\[\ph(x,t)=\sum_{l\in{\Bbb Z}}e^{-i\pi\dot{\eps}(l)}\;\psi(l+t-r)\;
e^{\pi i\la A(l)^\top\xi,x\ra}\]
whose formal derivatives are
\[(\partial_t^\alpha\partial_x^\beta\ph)(x,t)=\sum_{l\in{\Bbb Z}}e^{-i\pi\dot{\eps}(l)}\;
(\pi i)^{|\beta|}\;(A(l)^\top\xi)^\beta\;(\partial_t^\alpha\psi)(l+t-r)\;
e^{\pi i\la A(l)^\top\xi,x\ra}\;.\]
Since $B$ is nilpotent, the expression $A(l)^\top\xi=\exp(lB)^\top\xi$ is polynomial
in $l$. Hence for each multi-index $\beta$ there exist constants $N\in\NN$ and $C_0>0$
such that
\[|\,(A(l)^\top\xi)^\beta\,|\le C_0\,(\,1+l^2\,)^N\;.\]
for all $l\in\ZZ$. On the other hand, since $\psi\in\cS(\RR)$, for each $\alpha$
there are $C_1,C_2>0$ such that
\[|\,(\partial_t^\alpha\psi)(l+t-r)\,|\le C_1\,|\,1+(l+t-r)^2\,|^{-(N+1)}\le
C_2\,|\,1+l^2\,|^{-(N+1)}\]
for all $l$, and for $t$ ranging over a compact subset $K$ of $\RR$. This implies
that the above series converge absolutely and uniformly on $\RR^n\times K$ so that
$\ph\in C^\infty(G,\eps)$ is well-defined. Clearly $\ph\in U_\theta$ and
$T_\eta\ph=\psi$. \qed
\end{proof}

\subsection{Sub-Dirac operators with discrete spectrum}

Let $(H,\la\cdot,\cdot\ra)$ be a real vector space with inner product and
$(\Delta,\la\cdot,\cdot\ra)$ a complex vector space with a hermitian inner
product. Suppose that $\Delta$ carries a $\Cl(H)$-module structure such that
$\la x\cdot v,w\ra=-\la v,x\cdot w\ra$ for all $x\in H\subset\Cl(H)$ and
$v,w\in \Delta$. Let $s_1,\ldots,s_d$ be an orthonormal basis of $H$ and
$a\in H$ a non-zero multiple of $s_d$. Furthermore, let $\Omega:\RR\to\Span\{s_1,
\ldots,s_{d-1}\}$ be a non-constant polynomial function. We consider the
operator
\[P=a\,\del_t+i\,\Omega(t)\]
on the domain $\cS(\RR,\Delta)$. Here $a,\Omega(t)\in\Cl(H)$ are understood as
operators acting by pointwise multiplication. Clearly $P$ is symmetric with
respect to the $L^2$-inner product and densely defined in the Hilbert
space $L^2(\RR,\Delta)$. Thus $P$ is closable. The closure $\bar{P}$ of $P$
is a symmetric operator. Let $P^\ast$ denote the adjoint of $P$. On its domain
\[\dom(P^\ast)=\{\,\psi\in L^2(\RR,\Delta): \ph\mapsto\la P\ph,\psi\ra_{L^2}
\mbox{ is continuous w.r.t.\ to the }L^2\mbox{-norm}\,\}\]
we consider the norm $|\cdot|_P$ given by $|\,\psi\,|_P^2=|\,\psi\,|_{L^2}^2
+|\,P^\ast\psi\,|_{L^2}^2$. Our aim is to prove the following result.

\begin{pr}\label{pr:discrete_spectrum}
The operator $P$ is essentially self-adjoint and its closure $\bar{P}$ has
discrete spectrum.
\end{pr}
\begin{proof}
We can assume $|a|=1$ what will simplify the estimates below.

To prove the first assertion, we imitate the proof of the essential
selfadjointness of the Dirac operator, compare Theorem~5.7 of~\cite{LM}
and Proposition~1.3.5 of~\cite{G}. As a basic fact we know that it suffices
to verify $\ker(P^\ast\pm iI)=\{0\}$. Moreover, since $\bar{P}$ is symmetric,
it is enough to show that $\ker(P^\ast\pm iI)\subset\dom(\bar{P})$. To begin
with, we note that, if $f\in\cS(\RR)$ and $\psi\in\dom(P^\ast)$, then
$f\psi\in\dom(P^\ast)$ and
\[P^\ast(f\psi)=f(P^\ast\psi)+(\del_tf)a\cdot\psi\;.\]
Let $\psi\in\ker(P^\ast\pm iI)$. If $\hat{P}$ denotes the extension
of $P$ to tempered distributions, then we get, as $P$ is symmetric,
$(\hat{P}\pm iI)\psi=P^\ast\psi\pm i\psi=0$. Since the principal symbol
$p(\xi)=\xi\,a$ of $\hat{P}\pm iI$ is invertible for $\xi\neq 0$, the
regularity theorem for elliptic differential operators implies that $\psi$
is a smooth function. Choose $h\in C_0^\infty(\RR)$ satisfying $0\le h\le 1$
and $h(0)=1$, and put $h_k(t)=h(t/k)$ for $k\ge 1$. By definition $h_k\to 1$
and $h_k\psi\in C_0^\infty(\RR,\Delta)\subset\dom(P)$. Since
\[\Big|\,(\del_th_k)a\cdot\psi\,\Big|_{L^2}\le|\,\del_th_k\,|_\infty\,|\,a\,|\;
|\,\psi\,|_{L^2}\le\frac{1}{k}\,|\,\del_th\,|_\infty\,|\,\psi\,|_{L^2}\;,\]
it follows that
\begin{eqnarray*}
|\,\psi-h_k\psi\,|_P^2 &=& |\psi-h_k\psi|_{L^2}^2
+|P^\ast\psi-P^\ast(h_k\psi)|_{L^2}^2\\
&\le& |\psi-h_k\psi|_{L^2}^2+\left(\;|P^\ast\psi-h_k(P^\ast\psi)|_{L^2}
+|\,(\del_th_k)a\cdot\psi\,|_{L^2}\;\right)^2
\end{eqnarray*}
converges to $0$ for $k\to\infty$ by dominated convergence. Hence
$\psi\in\dom(\bar{P})$. This establishes the essential selfadjointness of $P$.

To prove that the spectrum of $\bar{P}$ is discrete, we need the following
two lemmata.
\begin{lm}\label{lm:cp_emb}
Let $(\Delta,\la\cdot,\cdot\ra)$ be a $\Cl(H)$-module as above and
$\Omega:\RR\to H$ a continuous function satisfying $|\,\Omega(t)\,|\to\infty$
for $|\,t\,|\to\infty$. Then 
\[X:=\{\,\ph\in L^2(\RR,\Delta):\del_t\ph\in L^2(\RR,\Delta)\mbox{ and }
\Omega\cdot\ph\in L^2(\RR,\Delta)\,\}\]
becomes a Hilbert space when endowed with the norm $||\,\ph\,||^2=|\ph|_{L^2}^2
+|\del_t\ph|_{L^2}^2+|\Omega\cdot\ph|_{L^2}^2$, and the inclusion
$X\to L^2(\RR,\Delta)$ is a compact operator.
\end{lm}
\begin{proof}
The first assertion is obvious. Let $(\ph_m)$ be a bounded sequence in $X$. We
prove that $(\ph_m)$ has a subsequence which is Cauchy in $L^2(\RR,\Delta)$. For
every $n\in\NN\setminus\{0\}$ there exists $r_n>0$ such that $|\Omega(t)|\ge n$
for $|t|\ge r_n$. We can assume $r_n<r_{n+1}$ and $r_n\to\infty$ for $n\to\infty$.
Using $|\Omega(t)\cdot\ph(t)|=|\Omega(t)|\,|\ph(t)|\ge n\,|\ph(t)|$, we obtain
\[\int_{|t|\ge r_n}|\ph(t)|^2\,dt\le\frac{1}{n^2}\int_{|t|\ge r_n}
|\Omega(t)\cdot\ph(t)|^2\,dt\le\frac{1}{n^2}\,|\,\Omega\cdot\ph\,|_{L^2}^2\]
for $\ph\in X$. For the moment, we fix the parameter $n$. Let $\chi\in C_0^\infty(\RR)$
be such that $0\le\chi\le 1$ and $\chi(t)=1$ for $|t|\le r_n$. By Rellich's theorem,
applied to the bounded sequence~$\chi\ph_m$ in~$H^1(\RR,\Delta)$, we conclude that
there exists a subsequence $(\ph_{m_{n,k}})$ such that
\[\int_{-r_n}^{r_n}|\ph_{m_{n,k}}(t)-\ph_{m_{n,l}}(t)|^2\,dt\to 0\]
for $k,l\to\infty$. Proceeding by induction, we establish $\{m_{n+1,k}:k\in\NN\}
\subset\{m_{n,k}:k\in\NN\}$ for all $n$. We define $m_k=m_{k,k}$. Now it is easy to
see that $(\ph_{m_k})$ is Cauchy w.r.t.\ the $L^2$-norm.\qed
\end{proof}

\begin{lm}\label{lm:dom_of_Dast}
The domain of $\bar{P}$ is contained in $X$ and the inclusion $\dom(\bar{P})\to X$
is continuous.
\end{lm}
\begin{proof}
Since $\dom(P)=\cS(\RR,\Delta)$ is contained in $X$ and dense in $\dom(\bar{P})$
w.r.t.\ the norm $|-|_P$, it suffices to prove that there exists $K>0$ such that
$||\,\ph\,||\le K\,|\,\ph\,|_{P}$ for all $\ph\in\cS(\RR,\Delta)$. Using
$a\cdot(\del_t\ph)=\del_t(a\cdot\ph)$ and $\Omega a=-a\Omega$
in $\Cl(H)$, we compute
\[\la\,a\cdot(\del_t\ph),i\Omega\cdot\ph\,\ra_{L^2}
=-\la\,a\cdot\ph,i\del_t(\Omega\cdot\ph)\,\ra_{L^2}\]
and
\[\la\,i\Omega\cdot\ph,a\cdot(\del_t\ph)\,\ra_{L^2}
=\la\,a\cdot\ph,i\Omega\cdot(\del_t\ph)\,\ra_{L^2}\;.\]
As $\del_t(\Omega\cdot\ph)=(\del_t\Omega)\cdot\ph+\Omega\cdot(\del_t\ph)$, it
follows
\[|\,P\ph\,|_{L^2}^2 = |a\cdot(\del_t\ph)+i\Omega\cdot\ph|_{L^2}^2
=|\,\del_t\ph\,|_{L^2}^2-\la\,a\cdot\ph,i(\del_t\Omega)\cdot\ph\,\ra_{L^2}
+|\,\Omega\cdot\ph\,|_{L^2}^2\]
for all Schwartz functions. Since $\Omega$ is a polynomial function, there exists
$r>0$ such that $|\,(\del_t\Omega)(t)\,|\le|\,\Omega(t)\,|$ for all $|t|\ge r$.
Fix $C>1$ such that $|\,(\del_t\Omega)(t)\,|\le C$ for $|t|\le r$. From
\begin{eqnarray*}
|\,\la\,a\cdot\ph,i(\del_t\Omega)\cdot\ph\,\ra_{L^2}\,| &\le&
\int_{-r}^r|\,(\del_t\Omega)(t)|\,|\ph(t)|^2\;dt
+\int_{|t|\ge r}|(\del_t\Omega)(t)|\,|\ph(t)|^2\;dt\\
&\le& C\int_{-r}^r|\ph(t)|^2\;dt+\int_{|t|\ge r}|\ph(t)|\,|\Omega(t)\cdot\ph(t)|\;dt\\
&\le& C\,|\,\ph\,|_{L^2}^2+|\,\ph\,|_{L^2}\,|\,\Omega\cdot\ph\,|_{L^2}
\end{eqnarray*}
it then follows
\begin{eqnarray*}
|\,\ph\,|_P^2 &=& |\,\ph\,|_{L^2}^2+|\,P\ph\,|_{L^2}^2\ge
\frac{1}{2C}\left(\,C\,|\,\ph\,|_{L^2}^2+\frac{1}{4}\,|\,\ph\,|_{L^2}^2
+|\,P\ph\,|_{L^2}^2\,\right)\\
&\ge& \frac{1}{2C}\left(\,\frac{1}{4}\,|\,\ph\,|_{L^2}^2+|\,\del_t\ph\,|_{L^2}^2
+|\,\Omega\cdot\ph\,|_{L^2}^2-|\,\ph\,|_{L^2}\,|\Omega\cdot\ph\,|_{L^2}\,\right)
\ge\frac{1}{2C}\,|\,\del_t\ph\,|_{L^2}^2
\end{eqnarray*}
for $\ph\in\cS(\RR,\Delta)$. Moreover, $i\Omega\cdot\ph=P\ph-a\cdot(\del_t\ph)$
gives
\[|\,\Omega\cdot\ph\,|_{L^2}\le |\,P\ph\,|_{L^2}+|\,\del_t\ph\,|_{L^2}
\le(1+\sqrt{2C})\,|\,\ph\,|_P\;.\]
Altogether, we obtain
\[||\,\ph\,||^2 = |\,\ph\,|_{L^2}^2+|\,\del_t\ph\,|_{L^2}^2+|\,\Omega\cdot\ph\,|_{L^2}^2
\le \left(\,1+2C+(1+\sqrt{2C})^2\,\right)|\,\ph\,|_P^2\]
proving the lemma.\qed
\end{proof}

Now we can prove the second assertion of Proposition~\ref{pr:discrete_spectrum}.
As $\sigma(\bar{P})\subset\RR$, we know that $\bar{P}-iI$ is bijective. Put
$R:=(\bar{P}-iI)^{-1}$. Since $\bar{P}-iI$ is continuous w.r.t.\ the complete
norm $|\,\cdot\,|_P$ on $\dom(\bar{P})$ and the $L^2$-norm, the open-mapping
theorem implies that $R:L^2(\RR,\Delta)\to\dom(\bar{P})$ is continuous. Moreover,
since the inclusion $\dom(\bar{P})\to X\to L^2(\RR,\Delta)$ is compact by
Lemma~\ref{lm:cp_emb} and~\ref{lm:dom_of_Dast}, it follows that $R$ is a
compact normal operator on $L^2(\RR,\Delta)$ with $\ker R=\{0\}$. By the
spectral theorem there exists an orthonormal basis $\{\ph_n:n\in\NN\}$
of $L^2(\RR,\Delta)$ with $R\ph_n=\mu_n\ph_n$ for suitable $0\neq\mu_n\in\CC$.
This implies $\bar{P}\ph_n=(\lambda+\frac{1}{\mu_n})\ph_n$ for all $n$. If
$\{\mu_n:n\in\NN\}$ happens to be an infinite set, then $\mu_n\to 0$ and
hence $|\lambda+\frac{1}{\mu_n}|\to\infty$ for $n\to\infty$. This proves
the spectrum of $\bar{P}$ to be discrete. \qed
\end{proof}

Now we resume the assumptions of Section~\ref{S3.1}: Let $A(t)=\exp(tB)$ be a
one-parameter group of $\GL(n,\RR)$ with $A(1)\in\SL(n,\ZZ)$. Suppose that $B$
does not possess any purely imaginary eigenvalues. Let $\fg$ denote the Lie algebra
of~$G$, and $\fn$ the Lie algebra of~$N=\RR^n\times\{0\}$. As before we identify
$\fg$ with the tangent space at the identity element $e=(0,0)$ of $G$ and denote
by $b$ the $(n+1)$th basis vector of $\fg\cong\RR^n\rtimes_B\RR$.

Let $\cH$ be a left-invariant oriented distribution on $G$ such that $b\in\cH_e$.
Suppose that $\cH$ carries a left-invariant Riemannian metric $g$ such that $b$ is
orthogonal to $\cH_e\cap\fn$ w.r.t.\ the inner product $\ip:=g_e$
on $\cH_e$. We assume that $\cH$ is bracket-generating. With $C^0(\cH_e)=\cH_e$
and $C^k(\cH_e)=[\cH_e,C^{k-1}(\cH_e)]$ for $k\ge 1$, this means
$\cH=\sum_{k=0}^n C^k(\cH_e)$. In particular, it follows
\begin{equation}\label{eq:weak_bg}
[\fg,\fg]=\sum_{k=1}^n C^k(\cH_e)\;.
\end{equation}
The latter condition is crucial for the proof of Theorem~\ref{theo:pure_point_spec}
but we do not claim that (\ref{eq:weak_bg}) has significance for left-invariant
distributions $\cH$ on Lie groups $G$ which do not have the form $G=\RR^n\rtimes_A\RR$.

Let $\nabla$ be a left-invariant metric connection on $\cH$ satisfying
condition~(\ref{Ediv}) of Lemma~\ref{lm:sym_of_D}, which guarantees the symmetry
of the sub-Dirac operator. Here the divergence in~(\ref{Ediv}) is defined w.r.t.\
the left-invariant volume form corresponding to the Haar measure of~$G$. If $\dot{\cH}$
is the distribution on~$\GG$ defined by~$\cH$, then the Riemannian metric on $\dot{\cH}$
induced by $g$ is denoted by $\dot{g}$, and the connection on $\dot{\cH}$ by $\dot{\nabla}$.
Let $\eps:\Gamma\to\ZZ_2$ be a homomorphism defining a spin structure
$P_{\Spin,\eps}(\dot{\cH})\cong G\times_\Gamma\Spin(d)$ of $(\dot{\cH},\dot{g})$,
where $d=\dim\cH$.

We fix a positively oriented orthonormal basis $s_1,\ldots,s_d$ of $\cH_e$ with
$s_1,\ldots,s_{d-1}\in\cH_e\cap\fn$ and such that $s_d$ is a positive multiple of~$b$.
Denoting the corresponding left-invariant vector fields again by $s_1,\ldots,s_d$,
the sub-Dirac operator, as acting on smooth sections of the spinor bundle
$S(\dot{\cH},\eps)$, is given by $D=\sum_i s_i\cdot\dot{\nabla}^S_{s_i}$.

Let $\Delta$ be the representation space of the complex spinor representation
of~$\Spin(\cH_e)$. Identifying $\Gamma(S(\dot{\cH},\eps))$ with
$C^\infty(G,\eps)\otimes\Delta$, we see that $D$ is given by
\begin{equation}\label{eq:expl_form_of_D}
D=\sum_i d\rho(s_i)\otimes s_i+\frac{1}{4}\sum_{i,j,k}\Gamma_{ij}^k\;
I\otimes s_is_js_k=:P+W
\end{equation}
where the $s_i$'s in the second factor of the tensor products are understood as
operators on~$\Delta$. Furthermore, the constants $\Gamma_{ij}^k=g(\nabla_{s_i}s_j,s_k)$
are the Christoffel symbols of~$\nabla$ w.r.t.\ the orthonormal frame $s_1,\ldots,s_d$
of $\cH$, and $d\rho$ is the derivative of the right regular representation $\rho$
of~$G$ on~$L^2(G,\eps)$. By~(\ref{Ediv}) the second sum in~(\ref{eq:expl_form_of_D})
reduces to a sum over all pairwise distinct indices $i,j,k$.

\begin{theo}\label{theo:pure_point_spec}
If, in addition to the above assumptions, $G=\RR^n\rtimes_A\RR$ is nilpotent,
then the closure of the operator $D$ on $\GG$ has a pure point spectrum.
\end{theo}
\begin{proof}
By Proposition~\ref{pr:L2_decomp} we know that $L^2(G,\eps)\otimes\Delta$ is a
direct sum of the orthogonal subspaces $\{U_\theta\otimes\Delta:\theta\in\ZZ\setminus
\Sigma_{\eps'}\}$ which are invariant under the action of~$\rho(G)\otimes\Cl(\cH_e)$.
Let $U_\theta^\infty:=U_\theta\cap C^\infty(G,\eps)$. Then $U_\theta^\infty\otimes\Delta$
is $D$-invariant. To prove that $\bar{D}$ has a pure point spectrum, it suffices to
prove that the closure of~$D_\theta:=D\,|\,U_\theta^\infty\otimes\Delta$ has a
pure point spectrum for all $\theta$. If $\theta=\{\xi\}$ is a fixed point, then,
according to Equation~(\ref{eq:rho_xi_for_fixed_points}), the subspace $U_\theta$
is an orthogonal sum of one-dimensional $\rho(G)$-invariant subspaces of~$U_\theta^\infty$.
Thus $U_\theta\otimes\Delta$ is an orthogonal sum of two-dimensional $D_\theta$-invariant
subspaces of $\dom(D_\theta)$ which shows that $D_\theta$ has a pure point
spectrum. Thus we are left with the case where $\theta$ is an infinite set
and $U_\theta$ is isomorphic to~$L^2(\RR)$. Fix $\xi\in\theta$. Lemma~\ref{lm:rho_eta}
implies that there exists a unitary isomorphism $T_\xi:U_\theta\to L^2(\RR)$ such
that $\rho_\xi:=T_\xi\,\rho_\theta\,T_\xi^\ast$ is given by
\begin{equation}\label{eq:rho_xi_again}
\rho_\xi(x,s)\psi\;(t)=e^{\pi i\la A(t)^\top\xi,x\ra}\;\psi(t+s).
\end{equation}
By Lemma~\ref{lm:range_of_T_eta} we may define $D_\xi=(T_\xi\otimes I)\,D_\theta\,
(T_\xi^\ast\otimes I)\,|\,\cS(\RR,\Delta)$. Since $d\rho_\xi(s_j)=\pi i\,\la
A(t)^\top\xi,s_j\ra$ for $1\le j\le d-1$ and $d\rho_\xi(s_d)=|\,b\,|\,\del_t$,
it follows that the operator $P_\xi:=\sum_{j=1}^d d\rho_\xi(s_j)\otimes s_j$
on $\cS(\RR,\Delta)$ has the form
\begin{equation}\label{eq:D_xi}
P_\xi=a\del_t+i\Omega_\xi(t)
\end{equation}
with $a=|\,b\,|^{-1}s_d$ and $\Omega_\xi(t)=\pi\,\sum_{j=1}^{d-1}\,\la
A(t)^\top\xi,s_j\ra\,s_j\in\cH_e\subset\Cl(\cH_e)$. Note that
$\la\,A(t)^\top\xi,s_j\,\ra=\la\xi,\exp(tB)s_j\,\ra$ is a polynomial
in $t$ because $B$ is nilpotent. Since
\[ [\fg,\fg]=\sum_{k=1}^nB^k(\cH_e\cap\fn)
=\sum_{k=1}^n\sum_{j=1}^{d-1}\,\RR\cdot(B^ks_j)\]
by~(\ref{eq:weak_bg}) and $\la\,\xi,[\fg,\fg]\,\ra\neq 0$ for non-fixed
points, it follows that at least one of the components of $\Omega_\xi$ is not
constant. Thus Proposition~\ref{pr:discrete_spectrum} implies that $\bar{P}_\xi$
has discrete spectrum. In other words, the essential spectrum of $\bar{P}_\xi$
is empty.

The operator $W_\xi:=\frac{1}{4}\sum_{i,j,k}\Gamma_{ij}^k\,s_is_js_k$
is bounded on $L^2(\RR,\Delta)$. In particular, $W_\xi$ is relatively
$\bar{P}_{\xi}$-compact in the sense that $\dom(\bar{P}_\xi)\subset\dom(W_\xi)$
and $W_\xi(\bar{P}_\xi-iI)^{-1}$ is compact. By the Kato-Rellich theorem we
know that $\bar{D}_\xi=\bar{P}_\xi+W_\xi$ is selfadjoint. Moreover, Weyl's
theorem which asserts the stability of the essential spectrum under relatively
compact perturbations, and for which we refer to Theorem~14.6 of~\cite{HS},
implies that the essential spectrum of $\bar{D}_\xi=\bar{P}_\xi+W_\xi$ is empty.
This shows that $\bar{D}_\xi$ and hence $\bar{D}_\theta$ have discrete spectrum.
The proof of the theorem is complete.\qed
\end{proof}

\begin{re}
In general, the eigenvalues of the sub-Dirac operator $D$ do not have finite
multiplicity and the spectrum of $D$ is not a discrete subset of $\RR$. A relevant
example is given in Section~\ref{S4.3}.
\end{re}

\subsection{Two- and three-dimensional distributions}

In this subsection we will compute the spectra of the operators $D_\theta$
arising in the proof of Theorem~\ref{theo:pure_point_spec} provided that
$G=\RR^n\rtimes_A\RR$ is 2-step nilpotent and $\dim\cH=2$ or~$3$. The explicit
formulas that will be given below in the non-fixed point case are a consequence
of the following result.

\begin{pr}\label{pr:spectrum_of_D}
Let $\alpha,\beta\in\RR$ and $\omega(t)=a\omega_1t+\omega_0$ where $a>0$,
$\omega_0,\omega_1\in\CC$ and $|\omega_1|=1$. Then the spectrum of the operator
\begin{equation}\label{eq:concrete_D}
D=\alpha I+\beta\,\left(\begin{array}{cc} i\del_t & \bar{\omega} \\ \omega & -i\del_t
\end{array}\right)
\end{equation}
on $\cS(\RR,\CC^2)$ is discrete. More precisely, $\sigma(D)=\{\,\lambda_0\,\}\cup
\{\,\lambda_{k}^{\pm}:k\in\NN\setminus\{0\}\,\}$, where
\[\lambda_0:=\alpha+\beta\Im(\omega_0\bar{\omega}_1)\quad{and}\quad
\lambda_k^{\pm}:=\alpha\pm\beta\left(2ak+\Im(\omega_0\bar{\omega}_1)^2\right)^{1/2}.\]
If $\lambda_0$ and the $\lambda_k^\pm$ are pairwise distinct, then all eigenvalues
are simple.
\end{pr}
\begin{proof}
We can assume $\alpha=0$ and $\beta=1$. Instead of $D$, we consider the operator
$S:=Q^*DQ$, where $Q$ is the unitary matrix
\[Q=\frac{1}{\sqrt{2}}\;\left(\begin{array}{cc} 1 & -i\bar{\omega}_1 \\
 -i\omega_1 & 1 \end{array}\right),\]
which diagonalizes $D^2$ and does not depend on $t$. Obviously the spectra
of $D$ and $S$ coincide. We have
\begin{eqnarray*}
S &=& \left(\begin{array}{cc} \Im(\omega_1\bar{\omega}) &
\bar{\omega}_1 \left(\del_t+\Re(\omega_1\bar{\omega})\,\right)\; \\
\;\omega_1\left(-\del_t+ \Re(\omega_1\bar{\omega})\,\right) & 
-\Im(\omega_1\bar{\omega})\end{array}\right)\\[0.4cm]
&=& \left(\begin{array}{cc} -\Im(\omega_0\bar{\omega}_1) &
\bar{\omega}_1\left(\,\del_t+a\,t+\Re(\omega_0\bar{\omega}_1)\,\right)\; \\
 \;\omega_1\left(\,-\del_t+a\,t+\Re(\omega_0\bar{\omega}_1)\right) &
\Im(\omega_0\bar{\omega}_1)\end{array}\right).
\end{eqnarray*}
To detect $S$-invariant subspaces, we start with the orthonormal basis
$\{h_k:k\in\NN\}$ of~$L^2(\RR)$ given by the (normalized) Hermite functions
\[h_k(x)=(\,2^k\,\pi^{1/2}\,k!\,)^{-1/2}\,H_k(x)\,e^{-x^2/2}.\]
Here $H_k(x)=(-1)^k\,e^{x^2}\del_x^k[e^{-x^2}]$ is the $k$th Hermite polynomial.
Put $b=2a\Re(\omega_0\bar\omega_1)$. Using the unitary isomorphism
\[(U\cdot w)(t):=a^{1/4}\,w(\,a^{1/2}(t+\frac{b}{2a^2})\,)\]
of $L^2(\RR)$, we then define $u_k=U\cdot h_k$. Recall that the creation
operator $\Lambda_+=-\del_x+x$ and the annihilation operator $\Lambda_-=\del_x+x$
satisfy $\Lambda_+(h_k)=\sqrt{2(k+1)}\,h_{k+1}$ and $\Lambda_-( h_k)=\sqrt{2k}\,h_{k-1}$.
Since
\[U\Lambda_\pm U^\ast=a^{-1/2}\,\left(\mp \del_t+a\,t+\frac{b}{2a}\right),\]
we thus obtain
\begin{eqnarray*}
\left(-\del_t+a\,t+\frac{b}{2a}\right)u_k 
&=& \sqrt{2a(k+1)}\;u_{k+1}\mbox{ for }k\ge 0,\\
\left(\del_t+a\,t+\frac{b}{2a}\right)u_k
&=& \sqrt{2ak}\;u_{k-1}\mbox{  for }k\ge 1,\\
\left(\del_t+a\,t+\frac{b}{2a}\right)u_0 &=& 0.
\end{eqnarray*}
This shows 
\begin{eqnarray*}
S\cdot\left(\begin{array}{c} 0 \\ u_0 \end{array}\right)
&=& \left(\begin{array}{c} \bar{\omega}_1\left(\del_t+a\,t
+\Re(\omega_0\bar{\omega}_1)\right)u_0 \\
\Im(\omega_0\bar{\omega}_1)\,u_0 \end{array}\right)
\ =\ \Im(\omega_0\bar{\omega}_1) \left(\begin{array}{c} 0 \\ u_0
\end{array}\right),\\[1ex]
S\cdot\left(\begin{array}{c} u_{k-1} \\ 0 \end{array}\right)
&=& \left(\begin{array}{cc} -\Im(\omega_0\bar{\omega}_1)\,u_{k-1}\\
\omega_1\left(-\del_t+a\,t+\Re(\omega_0\bar{\omega}_1)\right)u_{k-1}\end{array}\right)
\ =\ \left(\begin{array}{c} -\Im(\omega_0\bar{\omega}_1)\;u_{k-1} \\ \sqrt{2ak}\;
\omega_1\;u_{k}\end{array}\right),\\[1ex]
S\cdot\left(\begin{array}{c} 0 \\ u_k \end{array}\right)
&=& \left(\begin{array}{c} \bar{\omega}_1\left(\del_t+a\,t+\Re(\omega_0\bar{\omega}_1)
\right)u_k \\
\Im(\omega_0\bar{\omega}_1)\,u_k \end{array}\right)\ =\ \left(\begin{array}{c}
\sqrt{2ak}\;\bar{\omega}_1\;u_{k-1}\\ \Im(\omega_0\bar{\omega}_1)\;u_k \end{array}\right).
\end{eqnarray*}
In particular, the subspaces 
\[V_0:=\CC\left(\begin{array}{c} 0 \\ u_0\end{array}\right)\quad\mbox{and}\quad
V_k:=\Span\left\{\ph_k:=\left(\begin{array}{c} u_{k-1} \\ 0\end{array}\right),
\psi_k:=\left(\begin{array}{c} 0 \\ u_k\end{array}\right)\right\},\ k\ge1,\]
are $S$-invariant, and the restriction of $S$ to $V_k$, $k\ge 1$, is
given by the matrix 
$$\left(\begin{array}{cc} -\Im(\omega_0\bar \omega_1) & \sqrt{2ak}\,\bar\omega_1 \\ 
\sqrt{2ak}\,\omega_1 & \Im(\omega_0\bar\omega_1) \end{array}\right)$$
with respect to the basis $\ph_k,\psi_k$. Since $L^2(\RR,\CC^2)$ is the direct
sum of the $V_k$, $k\ge0$, the assertion follows.\qed
\end{proof}

Assume that $G=\RR^n\rtimes_A\RR$ is $2$-step nilpotent. Let $(\cH,g,\nabla)$ be
as in the preceding subsection with $2\le\dim\cH\le 3$. First we will determine
the spectrum of $D_\theta:=D\,|\,U_\theta^\infty\otimes\Delta$ when $\theta$
does not consist of a single point.

Suppose that $\dim\cH=2$. Let $s_1\in\cH_e\cap\fn$ and $s_2\in\cH_e$ be a positive
multiple of $b$ such that $s_1,s_2$ is a positively oriented orthonormal basis
of~$\cH_e$. In particular, $s_2=|\,b\,|^{-1}\,b$. By (\ref{Ediv}) we have
$\Gamma_{11}^1+\Gamma_{21}^2=0$ and $\Gamma_{12}^1+\Gamma_{22}^2=0$ which
implies that all Christoffel symbols of $\nabla$ vanish. Up to isomorphism, there
exists only one simple $\Cl(\cH_e)$-module. Let $\Delta=\CC^2$ be the one such
that $s_1$ and~$s_2$, represented as operators on $\Delta$, are given by
\[s_1=\left(\begin{array}{cc} 0 & -1 \\ 1 & 0 \end{array}\right)\quad\mbox{and}
\quad s_2=\left(\begin{array}{cc} i & 0 \\ 0 & -i \end{array}\right)\;.\]
Let $\theta\in\ZZ\setminus\Sigma_{\eps'}$ and $\xi\in\theta$ be a non-fixed
point. If $T_\xi:U_\theta\to L^2(\RR)$ is a unitary isomorphim as
in the proof of Theorem~\ref{theo:pure_point_spec} and
$D_\xi=(T_\xi\otimes I)D_\theta(T_\xi^\ast\otimes I)\,|\,\cS(\RR,\Delta)$,
then we know by~(\ref{eq:D_xi}) that $D_\xi$ has the form
\[D_\xi=|b|^{-1}\left(\begin{array}{cc} i\del_t & \bar{\omega}_\xi \\
\omega_\xi & -i\del_t\end{array}\right) \]
with $\omega_\xi(t)=\pi i\,|\,b\,|\,\la\,A(t)^\top\xi,s_1\,\ra$. Since $B^2=0$,
we have $A(t)=I+tB$ so that
\[\omega_\xi(t)=\pi i\,|\,b\,|\left(\,\la\,B^\top\xi,s_1\,\ra\,t
+\la\,\xi,s_1\,\ra\,\right)\]
is a non-constant affine-linear function. Thus Proposition~\ref{pr:spectrum_of_D}
implies that $\bar{D}_\xi$ has discrete spectrum. Moreover, the eigenvalues
of $D_\xi$ can be computed as follows: Put $\alpha=0$, $\beta=|\,b\,|^{-1}$,
$a=\pi\,|\,b\,|\,|\,\la B^\top\xi,s_1\ra\,|$, $\omega_1=i\sgn\la B^\top\xi,s_1\ra$
and $\omega_0=\pi i\,\la\xi,s_1\ra$. Note that $\Im(\omega_0\bar{\omega}_1)=0$.
Hence it follows that 
\begin{equation}\label{eq:lambda_for_2dim_H}
\lambda_0(\xi)=0\qquad\mbox{and}\qquad\lambda_k^\pm(\xi)
=\pm\,\left(\,2\pi\,|\,b\,|^{-1}\,|\la\,B^\top\xi,s_1\,\ra|\,k\right)^{1/2}
\end{equation}
with $k\in\NN\setminus\{0\}$ are the eigenvalues of $D_\xi$. This completes the
case $\dim\cH=2$.\\\\
Suppose that $\dim\cH=3$. Choose $s_1,s_2\in\cH_e\cap\fn$ and $s_3=|\,b\,|^{-1}b$
such that $s_1,s_2,s_3$ becomes a positively oriented orthonormal basis of $\cH_e$.
Up to isomorphism, there exist two simple $\Cl(\cH_e)$-modules. Let $\Delta=\CC^2$
be the one given by the representation
\[s_1=\left(\begin{array}{cc} 0 & i \\ i & 0 \end{array}\right)\quad,\quad
s_2=\left(\begin{array}{cc} 0 & -1 \\ 1 & 0 \end{array}\right)\quad,\quad
s_3=\left(\begin{array}{cc} i & 0 \\ 0 & -i \end{array}\right).\]
Note that $s_1s_2=s_3$ and $s_is_j+s_js_i=-2\delta_{ij}$ for all $1\le i,j\le 3$,
as operators on $\Delta$. Using this and that $\nabla$ is metric, we conclude that
the second sum in~(\ref{eq:expl_form_of_D}) simplifies to
$W=-\frac{1}{2}(\Gamma_{12}^3+\Gamma_{23}^1+\Gamma_{31}^2)\,I\otimes I$.

Let $\theta\in\ZZ\setminus\Sigma_{\eps'}$ and $\xi\in\theta$ be a non-fixed point.
If $T_\xi$ is a unitary isomorphism of $U_\theta$ onto $L^2(\RR)$ such that
$\rho_\xi=T_\xi\,\rho_\theta\,T_\xi^\ast$ is given by Equation~(\ref{eq:rho_xi}),
then the restriction $D_\xi$ of $(T_\xi\otimes I)\,D_\theta\,(T_\xi^\ast\otimes I)$,
when realized in $L^2(\RR,\Delta)$, to Schwartz functions has the form
\[D_\xi=-\frac{1}{2}(\Gamma_{12}^3+\Gamma_{23}^1+\Gamma_{31}^2)\,I+
|b|^{-1}\left(\begin{array}{cc} i\del_t & \bar{\omega}_\xi \\
\omega_\xi & -i\del_t\end{array}\right) \]
where $\omega_\xi$, as $G$ is $2$-step nilpotent, is a non-constant affine linear
function given by
\begin{align*}
\begin{split}
\omega_\xi(t) &= \pi i\,|\,b\,|\left(\,i\la\,A(t)^\top\xi,s_1\,\ra\,
+\la\,A(t)^\top\xi,s_2\,\ra\,\right)\\
&= -\pi\,|\,b\,|\left(\,\big(\la\,B^\top\xi,s_1\,\ra\,-i\la\,B^\top\xi,s_2\,\ra\,\big)\,t
+\la\,\xi,s_1\,\ra\,-i\la\,\xi,s_2,\ra\,\right)\,.
\end{split}
\end{align*}
Hence Proposition~\ref{pr:spectrum_of_D} implies that the eigenvalues of $D_\xi$
are
\begin{equation}\label{eq:lambda_0_for_3dim_H}
\lambda_0(\xi)=-\frac{1}{2}(\Gamma_{12}^3+\Gamma_{23}^1+\Gamma_{31}^2)
-\pi\;\frac{\la B^\top\xi,s_1\ra\la\xi,s_2\ra-\la\xi,s_1\ra\la B^\top\xi,s_2\ra}
{\left(\,\la B^\top\xi,s_1\ra^2+\la B^\top\xi,s_2\ra^2\,\right)^{1/2}}
\end{equation}
and
\begin{multline}\label{eq:lambda_k_for_3dim_H}
\lambda_k^{\pm}(\xi) = -\frac{1}{2}(\Gamma_{12}^3+\Gamma_{23}^1
+\Gamma_{31}^2)\pm\bigg(\,2\pi k\,|\,b\,|^{-1}\left(\,\la B^\top\xi,
s_1\ra^2+\la B^\top\xi,s_2\ra^2\,\right)^{1/2}\\
+\pi^2\;\frac{\left(\,\la B^\top\xi,s_1\ra\la\xi,s_2\ra
-\la\xi,s_1\ra\la B^\top\xi,s_2\ra\,\right)^2}{\la B^\top\xi,s_1\ra^2
+\la B^\top\xi,s_2\ra^2}\;\bigg)^{1/2}.
\end{multline}

In~(\ref{eq:lambda_for_2dim_H}) - (\ref{eq:lambda_k_for_3dim_H}) we rediscover
the fact that the eigenvalues $\lambda_k(\xi)$ do not depend on the choice
of the point $\xi$ on the orbit. More precisely, since $B^2=0$ and $A(t)B=B$,
it follows, in accordance with Lemma~\ref{lm:rho_eta}, that
$\lambda_k^\pm(A(t)^\top\xi)=\lambda_k^\pm(\xi)$ for all $t\in\RR$ and
$\xi\in\RR^n\setminus[\fg,\fg]^\bot$. This completes the case $\dim\cH=3$.\\\\
Finally we compute the spectrum of $D_\theta$ when $\theta=\{\xi\}$ is a
fixed point. For this purpose, we can drop the assumption that $G$ is
nilpotent.

By~(\ref{eq:rho_xi_for_fixed_points}) we know that $U_\theta$ is an
orthogonal sum of 1-dimensional subspaces $\{U_{\theta,k}:k\in\ZZ\}$
of~$U_\theta\cap C^\infty(G,\eps)$ on which $(x,s)\in G$ acts by
multiplication with
\[\chi_{\eps^\sharp,\theta,k}(x,s)=e^{\pi i\la\xi,x\ra}
e^{\pi i(2k+\dot{\eps}(1))s}\;.\]
Suppose that $\dim\cH=2$. Let $s_1,s_2$ and $\Delta$ be as above. In this case
the sub-Dirac operator reads $D=d\rho(s_1)\otimes s_1+d\rho(s_2)\otimes s_2$.
Since $d\chi_{\eps^\sharp,\theta,k}(s_1)=\pi i\,\la\,\xi,s_1\,\ra$ and
$d\chi_{\eps^\sharp,\theta,k}(s_2)=\pi i\,|\,b\,|^{-1}\,(2k+\dot{\eps}(1))$,
it follows that $D_{\theta,k}:=D_\theta\,|\,U_{\theta,k}\otimes\CC^2$ is
unitarily equivalent to
\begin{equation}\label{eq:D_theta_k}
D_{\xi,k}:=\alpha I+\beta\left(\begin{array}{cc} 2k+\dot{\eps}(1) &
\bar{\omega}_\xi \\ \omega_\xi & -2k-\dot{\eps}(1) \end{array} \right)
\end{equation}
on $\CC^2$, where $\alpha=0$, $\beta=-\pi\,|\,b\,|^{-1}$ and
$\omega_\xi=-i\,|\,b\,|\,\la\,\xi,s_1\,\ra$ are constants. Obviously,
$D_{\xi,k}$ admits the eigenvalues
\begin{equation}\label{eq:mu_k_for_2dim_H}
\mu_k^\pm(\xi)=\pm\,\pi\,\left(\,|\,b\,|^{-2}(2k+\dot{\eps}(1))^2
+\la\xi,s_1\ra^2\,\right)^{1/2}\;,\;k\in\ZZ\,.
\end{equation}
Suppose that $\dim\cH=3$. Let $s_1,s_2,s_3$ and $\Delta$ be as above. In this
case $D=P+W$ where $W=\alpha\,I\otimes I$ and $\alpha=-\frac{1}{2}(\Gamma_{12}^3
+\Gamma_{23}^1+\Gamma_{31}^2)$. Hence $D_{\theta,k}$ is unitarily equivalent
to $D_{\xi,k}$ as in~(\ref{eq:D_theta_k}) with $\beta=-\pi\,|\,b\,|^{-1}$
and $\omega_{\xi}=|\,b\,|\,(\,\la\,\xi,s_1\,\ra-i\la\,\xi,s_2\,\ra\,)$. Thus
$D_{\xi,k}$ has the eigenvalues
\begin{equation}\label{eq:mu_k_for_3dim_H}
\mu_k^\pm(\xi)=-\frac{1}{2}(\Gamma_{12}^3+\Gamma_{23}^1+\Gamma_{31}^2)
\,\pm\,\pi\,\left(\,|\,b\,|^{-2}(2k+\dot{\eps}(1))^2+\la\xi,s_1\ra^2
+\la\xi,s_2\ra^2\,\right)^{1/2}.
\end{equation}
This shows that $D_\theta$ has discrete spectrum in the fixed point case.

\section{Examples of spectra of sub-Dirac operators}

\subsection{A preliminary remark}
To compute the spectrum of the sub-Dirac operator $D$, it remains, by the
results in the preceding section for the spectra of the $D_\theta$, to
determine a set of representatives for the set of all $\ZZ$-orbits
contained in $\Sigma_{\eps'}$. More precisely, in view of
Theorem~\ref{theo:decomp}, we carry out the following steps:
\begin{enumerate}
\item Describe all homomorphisms $\eps:\Gamma\to\ZZ_2$.
\item Find a set of representatives $\cR_{\eps'}$ for all $\RR$-orbits $\omega$
intersecting $\Sigma_{\eps'}$.
\item Compute the number of $\ZZ$-orbits contained in $\omega\cap\Sigma_{\eps'}$.
\item Determine the spectrum of $D_\xi$ for some $\xi\in\omega$.
\end{enumerate}   
This requires a detailed knowledge of the orbit picture of the coadjoint representation.
In the following examples, the eigenvalues of the sub-Dirac operator including their
multiplicities will be determined completely.

\subsection{Three-dimensional Heisenberg manifolds}\label{S4.2}

As we will see next, the results of this section comprise Theorem~3.1
of~{\rm \cite{AB}} concerning the spectrum of the Dirac operator on
three-dimensional Heisenberg manifolds.

Let $G=\RR^2\rtimes_A\RR$ be the Heisenberg group as discussed
in Example~\ref{ex:3_dim_Heis} and $\Gamma=\ZZ^2\rtimes_A\ZZ$. Then
$\fg=\Span\{e_1,e_2,b\}$ with $[b,e_2]=re_1$.
For positive real numbers $d$ and~$T$ we consider the orientation and
the Riemannian metric $g$ on $\cH:=TG$ such that $s_1=\frac{1}{T}e_1$,
$s_2=-de_2$ and $s_3=\frac{d}{r}b$ becomes a positively oriented orthonormal
frame. The constants are chosen in accordance with~\cite{AB}, where the
collapse of Heisenberg manifolds $M(r,d,T)$ for $T\to 0$ is studied. Let
$\nabla$ be the Levi-Civita connection of~$g$. In particular, $\nabla$
satisfies~(\ref{Ediv}) and $-\Gamma_{12}^3=\Gamma_{23}^1=\Gamma_{31}^2=\frac{d^2T}{2}$.
Let $\eps':\ZZ^2\to\ZZ_2$ be a homomorphism. We abbreviate $\eps'(e_\mu)$ to $\eps_\mu$.
Then (\ref{EsumA}) is satisfied if and only if $\eps_1r$ is even. It is easy to see
that the disjoint union $\cR_{\eps'}$ of
\[\cR_{\eps'}^{(1)} = \{\,\xi\in\Sigma_{\eps'}:\xi_1=0\,\}\quad\mbox{and}\quad
\cR_{\eps'}^{(2)} = \{\,\xi\in\Sigma_{\eps'}:\xi_1\neq 0\mbox{ and
}\xi_2=\eps_2\,\}\]
is a set of representatives for the set of all $\RR$-orbits intersecting
$\Sigma_{\eps'}$. The set $\cR_{\eps'}^{(1)}$ consists of all fixed
points in $\Sigma_{\eps'}$. If $\omega$ is the $\RR$-orbit represented
by $(\xi_1,\eps_2)\in\cR_{\eps'}^{(2)}$, then
$\omega\cap\Sigma_{\eps'}$ contains $|\xi_1r|/2$ distinct
$\ZZ$-orbits.

We compute $|\,b\,|=\frac{r^2}{d^2}$, $\la\xi,s_1\ra=\frac{1}{T}\,\xi_1$,
$\la\xi,s_2\ra=-d\xi_2$, $\la B^\top\xi,s_1\ra=0$ and
$\la B^\top\xi,s_2\ra=-dr\,\xi_1$. Inserting this
into~(\ref{eq:mu_k_for_3dim_H}) gives
\begin{equation}\label{1*}
\mu_k^\pm(0,\xi_2)=-\frac{d^2T}{4}\pm\pi\left(\,(2k+\dot\eps(1))^2\frac{d^2}{r^2}
+d^2\xi_2^2\,\right)^{1/2}.
\end{equation}
Similarly, for $\xi\in\Sigma_{\eps'}$ with $\xi_1\neq 0$,
Equations~(\ref{eq:lambda_0_for_3dim_H})
and~(\ref{eq:lambda_k_for_3dim_H}) yield
\begin{equation}\label{2*}
\lambda_0(\xi_1,\eps_2)=-\frac{d^2T}{4}-\frac{\pi}{T}\,|\xi_1|
\end{equation}
and
\begin{equation}\label{3*}
\lambda_k^{\pm}(\xi_1,\eps_2)=-\frac
{d^2T}{4}\pm\left(\,2\pi d^2k|\xi_1|
+\frac{\pi^2}{T^2}\,\xi_1^2\right)^{1/2}\,.
\end{equation}

We will use the following notation for the description of the spectrum of the
Dirac operator $D$ on $\GG$. We define the spectral multiplicity function
$$m(D):\RR\to\NN\cup\{|\NN|\},\quad m(D)(\lambda)=\dim\ker(D-\lambda I).$$
Moreover, $\delta=\delta(\lambda)$ denotes the function that
takes the value $1$ in $\lambda$ and that is zero on $\RR\setminus\{\lambda\}$.

Suppose we are given a spin structure corresponding to a homomorphism
$\eps:\ZZ^2\rtimes_A\ZZ\rightarrow \ZZ_2$ with $\eps_1=0$. Summation of (\ref{1*}),
(\ref{2*}) and (\ref{3*}) over $\xi_\nu\in 2\ZZ+\eps_\nu$ gives
$$m(D)=m_1^+(D)+m_1^-(D)+m_2^+(D)+m_2^0(D)+m_2^-(D),$$
where
\begin{eqnarray*}
m_1^\pm(D) &=& \sum_{k\in\Bbb Z}\sum_{l\in\Bbb
Z}\,\delta\left(\,-\frac{d^2T}{4}\pm
\frac{\pi d}{r}\left(\,(2k+\dot\eps(1))^2+r^2(2l+\eps_2)^2\,\right)^{1/2}\,\right)\,,\\
m_2^0(D) &=& \sum_{l=1}^\infty\,2rl\,\delta\left(\,-\frac{d^2T}{4}-\frac{2\pi
l}{T}
\,\right)\,,\\
m_2^\pm(D) &=& \sum_{l=1}^\infty\,2rl\,\sum_{k=1}^\infty\delta\left(\,-\frac{d^2T}{4}
\pm\left(\,4\pi d^2kl+\frac{4\pi^2l^2}{T^2}\,\right)^{1/2}\,\right)\;.
\end{eqnarray*}
Now assume $\eps_1=1$, which is only possible if $r$ is even. Then
$\cR_{\eps'}^{(1)}=\emptyset$ and we obtain
$$m(D)=m_2^+(D)+m_2^0(D)+m_2^-(D),$$ where now
\begin{eqnarray*}
m_2^0(D) &=& \sum_{l=0}^\infty\,(2l+1)r\,\delta\left(\,-\frac{d^2T}{4}-\frac{\pi(2l+1)}{T}
\,\right)\,,\\
m_2^\pm(D) &=& \sum_{l=0}^\infty\,(2l+1)r\,\sum_{k=1}^\infty\delta\left(\,-\frac{d^2T}{4}
\pm\left(\,2\pi d^2k(2l+1)+\frac{\pi^2(2l+1)^2}{T^2}\,\right)^{1/2}\,\right)\;.
\end{eqnarray*}

Now let us turn to the sub-Riemannian case and suppose
$\cH=\Span\{s_2,s_3\}$, where again $s_2$ and $s_3$ are orthonormal.
Let $\nabla$ be defined by (\ref{EK}) for a leftinvariant complement
$\cV:=\RR\cdot u$, $u\in\fg$ of $\cH$. Since we wish to get a
symmetric sub-Dirac operator, the only possible choice is
$\cV:=\RR\cdot s_1$. Indeed, otherwise
$[\Gamma(\cH),u]\not\subset\Gamma(\cH)]$, thus $D$ is not symmetric by
Lemma~\ref{lm:sym_of_D}. We proceed as above, now using
Equations~(\ref{eq:mu_k_for_2dim_H}) and~(\ref{eq:lambda_for_2dim_H}).

For a spin structure that corresponds to a homomorphism
$\eps:\ZZ^2\rtimes_A\ZZ\rightarrow \ZZ_2$ with $\eps_1=0$ we obtain
$$m(D)=|\NN|\cdot \delta(0) +m_1^+(D)+m_1^-(D)+m_2^+(D)+m_2^-(D),$$ where
\begin{eqnarray*}
m_1^\pm(D) &=& \sum_{k\in\Bbb Z}\sum_{l\in\Bbb Z}\,\delta\left(\,\pm
\frac{\pi d}{r}\left(\,(2k+\dot\eps(1))^2+r^2(2l+\eps_2)^2\,\right)^{1/2}\,\right)\,,\\
m_2^\pm(D) &=& \sum_{l=1}^\infty\,2rl\,\sum_{k=1}^\infty\delta\left(\,
\pm\left(\,4\pi d^2kl\,\right)^{1/2}\,\right)\;.
\end{eqnarray*}
If $\eps_1=1$, then
\begin{eqnarray*}
\lefteqn {m(D)=|\NN|\cdot \delta(0)\ +
}\\&&\sum_{l=0}^\infty\,(2l+1)r\,\sum_{k=1}^\infty\left(\delta\big(\,
(\,2\pi d^2k(2l+1))^{1/2}\,\big)+\delta\big(
-(\,2\pi d^2k(2l+1))^{1/2}\big)\,\right) .
\end{eqnarray*}

\subsection{A five-dimensional two-step nilpotent example}
\label{S4.3}

We start by considering 2-step nilpotent Lie groups which are isomorphic
to a standard model $G=\RR^{2p}\rtimes_A \RR$ as described in
Lemma~\ref{lm:standard_2_step_nilp} and therefore generalise the example
from the preceding subsection. We will describe the orbits of $\RR$
and $\ZZ$ acting on $\RR^{2p}$ by $A^\top$. Then we will specialise to
$\dim G=5$ for the computation of the spectrum of the sub-Dirac
operator on $\GG$.

\begin{lm}\label{lm:standard_2_step_nilp}
Let $G$ be a simply connected Lie group satisfying $[G,G]=Z(G)$ and
admitting a connected abelian normal subgroup $N$ of codimension~$1$.
Let $\Gamma$ be a uniform discrete subgroup of $G$ such that $\Gamma\cap N$
is uniform in~$N$. Then there exist $p\ge 1$, a one-parameter subgroup 
of $\GL(2p,\RR)$ of the form
\[A(t)=\left( \begin{array}{cc} I & tR \\ 0 & I \end{array} \right)\]
with $R=\diag(r_1,\ldots,r_p)$ and positive integers $r_\nu$ such that
$r_{\nu+1}\,|\,r_\nu$ for $\nu=1,\ldots,p-1$, and an isomorphism $\Phi$ of $G$
onto $\RR^{2p}\rtimes_A\RR$ such that $\Phi(\Gamma)=\ZZ^{2p}\times_A\ZZ$.
\end{lm}
\begin{proof}
Put $p=\dim Z(G)=\frac{1}{2}\dim N$. Since $\Gamma\cap Z(G)$ is uniform
in $Z(G)$, we find generators $v_1,\ldots,v_{2p}$ of $\Gamma\cap N$ such
that $\Gamma\cap Z(G)=\ZZ v_1+\ldots+\ZZ v_p$. As in the proof of
Lemma~\ref{lm:model}, we consider the linear isomorphism $M:\RR^{2p}\to N$
given by $M(e_\nu)=v_\nu$, choose $b\in\fg$ such that $\exp(b)\in\Gamma$
and $\exp(b)N$ generates $\Gamma N/N$, and define $A_0(t)\in\GL(2p,\RR)$
such that $\Phi_0(x,t)=M(x)\exp(tb)$ becomes an isomorphism of~$G$
onto~$\RR^{2p}\rtimes_{A_0}\RR$. Since
$\Phi_0(Z(G))=\RR^p\times\{0\}\times\{0\}$, we have
\[A_0(t)=\left(\begin{array}{cc} I & tR_0 \\ 0 & I\end{array}\right)\]
with $R_0\in\GL(p,\ZZ)$. Recall that $R_0$ can be brought into Smith normal
form, i.e., there exist $Q_1,Q_2\in\GL(p,\ZZ)$ such that $R:=Q_1R_0Q_2^{-1}=
\diag(r_1,\ldots,r_p)$ is diagonal with positive integers $r_\nu$ such that
$r_{\nu+1}\,|\,r_\nu$. Clearly $\Psi(x',x'',t):=(Q_1x',Q_2x'',t)$ gives
an isomorphism of~$\RR^{2p}\rtimes_{A_0}\RR$ onto~$\RR^{2p}\rtimes_{A}\RR$
with $A(t)(x',x'')=(x'+tRx'',x'')$. Finally, it follows that the assertion
of the lemma holds for $\Phi:=\Psi\Phi_0$. \qed
\end{proof}

Let $G=\RR^{2p}\rtimes_A\RR$ be as in Lemma~\ref{lm:standard_2_step_nilp} with
uniform discrete subgroup $\Gamma=\ZZ^{2p}\rtimes_A\ZZ$. In particular,
$A(t)e_\nu=e_\nu$ and $A(t)e_{p+\nu}=e_{p+\nu}+r_\nu t\,e_\nu$ for all
$1\le\nu\le p$.

Let $\dot{\eps}:\ZZ\to\ZZ_2$ and $\eps':\ZZ^{2p}\to\ZZ_2$ be homomorphisms.
As before, we put $\eps_\nu=\eps'(e_\nu)$ for $\nu=1,\dots,2p$. By
Lemma~\ref{lm:eps} it follows that $\eps(k,l):=\eps'(k)+\dot{\eps}(l)$
is a homomorphism of~$\Gamma$ if and only if $r_\nu\,\eps_\nu\in 2\ZZ$
for $1\le\nu\le p$. Note that the latter condition implies $\eps_\nu=0$
whenever $r_\nu$ is odd.

Next we will describe the coadjoint orbits. First of all,
\begin{equation}\label{eq:coadj_rep}
\la A(t)^\top\xi,e_\nu\ra=\xi_\nu\quad\text{and}\quad
\la A(t)^\top\xi,e_{p+\nu}\ra=\xi_{p+\nu}+r_\nu\xi_\nu\,t
\end{equation}
for $1\le\nu\le p$. To formulate the subsequent result, a little more notation
is needed. If $\xi\in\ZZ^{2p}$, then $\bar{\xi}\in\ZZ^p$ denotes the projection
of $\xi$ onto the first $p$ variables. For $\eta\in\ZZ^p$, the subset
$\{\,\xi\in\ZZ^{2p}:\bar{\xi}=\eta\,\}$ is $\ZZ$-invariant. In particular,
$\{\,\xi:\bar{\xi}=0\,\}$ is the set of all points remaining fixed
under the coadjoint action. Put $J_\eta=\{\,\nu:\eta_\nu\neq 0\,\}$.
For $\eta\neq 0$, let $d_\eta>0$ be the greatest common divisor of the integers
$|r_1\,\eta_1|,\ldots,|r_p\eta_p|$. We choose $j_\eta=\min J_\eta$
and set $q_\eta=|r_{j_\eta}\eta_{j_\eta}|\,/\,d_\eta$.

Let $\bar{\Sigma}_{\eps'}$ be the image of $\Sigma_{\eps'}$ under projection.
If $\eta\in\bar{\Sigma}_{\eps'}\setminus\{0\}$, then $d_\eta$ is even
because $\eta_{\nu}$ is even whenever $r_\nu$ is odd. Furthermore, we define
$\cR_{\eps'\!,0}=\{\,\xi\in\Sigma_{\eps'}:\bar{\xi}=0\,\}$ and
$\cR_{\eps'\!,\eta}=\{\,\xi\in\Sigma_{\eps'}:\bar{\xi}=\eta$ and
$0\le \xi_{p+j_\eta}\le 2q_\eta-1\,\}$ for $\eta\in\bar{\Sigma}_{\eps'}$
non-zero. Note that $\cR_{\eps'\!,0}$ is empty if $\eps_\nu=1$ for
some $1\le\nu\le p$.

\begin{lm}\label{lm:set_of_rep}
In this situation, the following holds true:
\begin{enumerate}[label=(\roman*)]
\item The disjoint union $\cR_{\eps'}:=\bigcup_{\eta\in\bar{\Sigma}_{\eps'}}
\cR_{\eps'\!,\eta}$ is a set of representatives for the set of all $\RR$-orbits
intersecting $\Sigma_{\eps'}$.
\item Let $\omega$ be an $\RR$-orbit which intersects $\Sigma_{\eps'}$. Then
$\eta:=\bar\xi$ does not depend on the choice of $\xi\in\omega\cap\Sigma_{\eps'}$.
If $\omega$ is not a fixed point, then $\omega\cap\Sigma_{\eps'}$ consists
of~$d_{\eta}/2$ distinct $\ZZ$-orbits.
\end{enumerate}
\end{lm}
\begin{proof}
Let $\xi\in\Sigma_{\eps'}$ such that $\bar{\xi}\neq 0$. By (\ref{eq:coadj_rep})
we know that $A(t)^\top\xi\in\Sigma_{\eps'}$ if and only if
$r_\nu\xi_{\nu}t\in 2\ZZ$ for all $\nu\in J_{\bar{\xi}}$. This proves
\begin{equation}\label{eq*}\{\,t\in\RR:A(t)^\top\xi\in\Sigma_{\eps'}\,\}
\;=\;\bigcap_{\nu\in J_{\bar{\xi}}}\frac{2}{|r_\nu\xi_{\nu}|}\,\ZZ\;
=\;\frac{2}{d_{\bar{\xi}}}\,\ZZ
\end{equation}
To prove \textit{(i)}, let $\omega$ be an $\RR$-orbit and
$\xi\in\omega\cap\Sigma_{\eps'}$. Clearly $\eta:=\bar{\xi}$ does not depend on
the choice of $\xi$. We can assume $\bar{\xi}\neq 0$. Define $d_\eta$ and
$j=j_\eta$ as above. Since $\la A(t)^\top\xi,e_{p+j}\ra=e_{p+j}+r_j\xi_{j}\,t$,
it follows from (\ref{eq*}) that there exists
$t\in\frac{2}{d_\eta}\ZZ$ such that
$A(t)^\top\xi\in\Sigma_{\eps'}$ and $0\le\la A(t)^\top\xi,e_{p+j}\ra\le
2q_\eta-1$.
This proves $A(t)^\top\xi\in\omega\cap\cR_{\eps'\!,\eta}$ because
$\overline{A(t)^\top\xi}=\bar{\xi}$.
We claim that $\omega\cap\cR_{\eps'\!,\eta}$ consists of a single point: If
$\xi,\xi^\ast\in\omega\cap\cR_{\eps'\!,\eta}$, then, again by~(\ref{eq*}),
there exists a $t\in\frac{2}{d_\eta}\,\ZZ$ such that $\xi^\ast=A(t)^\top\xi$.
In particular $\xi^\ast_{p+j}=\xi_{p+j}+r_j\xi_{j}t$. Since
$0\le\xi_{p+j},\xi^\ast_{p+j}\le 2q_\eta-1$, it follows $t=0$ and hence
$\xi^\ast=\xi$. This proves $\cR_{\eps'}$ to be a set of representatives.

Let $\xi\in\omega\cap\Sigma_{\eps'}$ be an arbitary non-fixed point. Then
$f:\RR\to\omega$, $f(t)=A(t)^\top\xi$, is bijective and $\RR$-equivariant.
By~(\ref{eq*}) it holds $f^{-1}(\omega\cap\Sigma_{\eps'})=\frac{2}{d_{\eta}}\,\ZZ$.
Since $d_{\eta}/2$ is an integer, it follows
\[\#\;\ZZ\setminus\omega\cap\Sigma_{\eps'}=\#\;\ZZ\setminus\frac{2}{d_{\eta}}
\ZZ=\frac{d_{\eta}}{2}\;.\]
More precisely, the points $\{A(\frac{2k}{d_{\eta}})^\top\xi:
0\le k<\frac{d_{\eta}}{2}\}$ are representatives for the set of all
$\ZZ$-orbits in $\omega\cap\Sigma_{\eps'}$.
\qed
\end{proof}

We point out that the choice of the set $\cR_{\eps'}$ is in no way canonical.
For example, any choice of indices $j_\eta\in J_\eta$ leads to a set of
representatives.\\\\
Now let us restrict ourselves to $p=2$. Then canonical basis $e_1,\ldots,e_4,b$
of the Lie algebra $\fg\cong\RR^{2p}\rtimes_B\RR$ of $G$ satisfies the relations
$[b,e_3]=r_1e_1$ and $[b,e_4]=r_2e_2$.

Put $s_1=e_3$, $s_2=e_4$ and $s_3=b$. As before, the corresponding left-invariant
vector fields are denoted by the same symbol. The left-invariant
distribution $\cH:=\Span\{s_1,s_2,s_3\}$ is given the orientation
and Riemannian metric $g$ such that $s_1,s_2,s_3$ is a positively oriented,
orthonormal frame. In particular, $|\,b\,|=1$. Note that $\cH$ is
bracket-generating.
\begin{re}
In general, when $\cH$ is a left-invariant 3-dimensional distribution on a
Lie group $G$, the affine space of all left-invariant metric connections in $\cH$
satisfying~(\ref{Ediv}) has dimension~6. However, in the present example, the
left-invariant connections which are defined by a left-invariant projection~$\proj$
onto~$\cH$ and the Koszul formula~(\ref{EK}) and which satisfy~(\ref{Ediv}) form
a $3$-dimensional space.
\end{re}
Let $\nabla$ be a left-invariant metric connection in $\cH$ satisfying~(\ref{Ediv}).
For example, we could take the connection given by projection onto $\cH$
along~$\cV:=\Span\{e_1,e_2\}$, which, according to~(\ref{EK}), satisfies
$\Gamma_{ij}^k=0$ for all $i,j,k$ because $[\fg,\fg]\subset\cV$.
Let $\eps:\Gamma\to\ZZ_2$ be a homomorphism giving a spin structure of $\dot{\cH}$.
By Lemma~\ref{lm:sym_of_D} the sub-Dirac operator $D$ defined
by $(\cH,g,\nabla,\eps)$ is symmetric. We compute its spectrum. To this end,
we note that the coadjoint representation is given by
\[B^\top\xi=\left(\begin{array}{c} 0 \\ 0 \\ r_1\xi_1 \\ r_2\xi_2
\end{array}\right)\qquad\mbox{and}\qquad  A(t)^\top\xi=
\left(\begin{array}{c} \xi_1 \\ \xi_2 \\ \xi_3+r_1\xi_1t \\ \xi_4+r_2\xi_2t
\end{array}\right)\;.\]
In particular, we get $\la\xi,s_\nu\ra=\la\xi,e_{2+\nu}\ra=\xi_{2+\nu}$
and $\la B^\top\xi,s_\nu\ra=r_\nu\xi_{\nu}$ for $\nu=1,2$. Put
$\alpha=-\frac{1}{2}\,(\,\Gamma_{12}^3+\Gamma_{23}^1+\Gamma_{31}^2\,)$.
By (\ref{eq:mu_k_for_3dim_H}), (\ref{eq:lambda_0_for_3dim_H}) and
(\ref{eq:lambda_k_for_3dim_H}), the eigenvalues of $D_\xi$ are
of the form
\[\mu_k^\pm(\xi)=\alpha\pm\pi\left(\,(2k+\dot{\eps}(1))^2+\xi_3^2+\xi_4^2\,
\right)^{1/2}\]
for fixed points, and
\[\lambda_0(\xi)=\alpha-\pi\,\frac{r_1\xi_1\xi_4-r_2\xi_2\xi_3}
{(\,r_1^2\xi_1^2+r_2^2\xi_2^2\,)^{1/2}}\]
or
\[\lambda_k^\pm(\xi)=\alpha\pm\left(\,2\pi k(r_1^2\xi_1^2+r_2^2\xi_2^2)^{1/2}+\pi^2\,
\frac{(\,r_1\xi_1\xi_4-r_2\xi_2\xi_3\,)^2}{r_1^2\xi_1^2+r_2^2\xi_2^2}\,\right)^{1/2}\]
else. We want to decompose the set $\cR_{\eps'}$ of representatives
into a disjoint union of sets that we can describe explicitly. To this
end, consider $\eta=\bar\xi \in \bar{\Sigma}_{\eps'}$ and assume
$\eta\not=0$. If $\eta_1\not=0$ and $\eta_2=0$, then
 $j_\eta=1$, $d_\eta=|r_1\eta_1|$ and $q_\eta=1$.
Similarly, if $\eta_1=0$ and $\eta_2\not=0$, then $j_\eta=2$,
$d_\eta=|r_2\eta_2|$ and $q_\eta=1$. For $\eta_1\eta_2\not=0$ we get
 $j_\eta=1$, and obtain $d_\eta=\gcd(|\,r_1\eta_1|,|r_2\eta_2|\,)$
and $q_\eta=|r_1\eta_1|\,/\,d_\eta$. This leads to a decomposition
of~$\cR_{\eps'}$ into the following subsets:
\begin{eqnarray*}
\cR_{\eps'}^{(1)} &=& \cR_{\eps',0}\,,\\
\cR_{\eps'}^{(2)} &=& \{\,\xi\in\Sigma_{\eps'}:\xi_1\neq 0,\,\xi_2=0,\,
\xi_3=\eps_3\,\}\,,\\
\cR_{\eps'}^{(3)} &=& \{\,\xi\in\Sigma_{\eps'}:\xi_1=0,\, \xi_2\neq 0,\,
\xi_4=\eps_4\,\}\,,\\
\cR_{\eps'}^{(4)} &=& \{\,\xi\in\Sigma_{\eps'}:\xi_1\neq 0,\,\,\xi_2\neq 0,\,
0\le \xi_3\le 2q_{(\xi_1,\xi_2)}-1\,\}\,.
\end{eqnarray*}
We have $\cR_{\eps'}^{(1)}=\emptyset$ if $\eps_1=1$ or $\eps_2=1$,
$\cR_{\eps'}^{(2)}=\emptyset$ if $\eps_2=1$, and $\cR_{\eps'}^{(3)}=\emptyset$
if $\eps_1=1$. Recall that for $\nu=1,2$ the case $\eps_\nu=1$ can occur only
if $r_\nu$ is even.

The spectrum of $D$ depends on the spin structure given by $\eps$. It holds
$m(D)=\sum_{i=1}^4m_i$ where $m_i=\sum_{\xi\in\cR_{\eps'}^{(i)}}m(D_\xi)$.
Note that $m_i=0$ if $\cR_{\eps'}^{(i)}=\emptyset$. Otherwise, $m_i$ is given
as follows, where the sums are meant to be taken over $\xi_\nu\in 2\ZZ+\eps_\nu$
and $\xi_5\in 2\ZZ+\dot{\eps}(1)$.
\begin{eqnarray*}
m_1 &=& \sum_{\xi_3,\xi_4,\xi_5}\,\delta(\,\alpha+\pi(\,\xi_3^2+\xi_4^2
+\xi_5^2\,)^{1/2}\,)+\delta(\,\alpha-\pi(\,\xi_3^2+\xi_4^2+\xi_5^2\,)^{1/2}\,)\\
m_2 &=& \sum_{\xi_1\neq 0}\frac{|r_1\xi_1|}{2}\sum_{\xi_4}
\bigg(\,\delta(\,\alpha-\pi\sgn(r_1\xi_1)\xi_4\,)\\
& & \hspace{0.5cm}+\sum_{k=1}^\infty\,\left(\,\delta(\,\alpha+(2\pi k|r_1\xi_1|
+\pi^2\xi_4^2)^{1/2}\,)+\delta(\,\alpha-(2\pi k|r_1\xi_1|+\pi^2\xi_4^2)^{1/2}\,)
\,\right)\,\bigg)\\
m_3 &=&\sum_{\xi_2\neq 0}\frac{|r_2\xi_2|}{2}\sum_{\xi_3}\bigg(\,\delta(\,
\alpha+\pi\sgn(r_2\xi_2)\xi_3\,)\\
& & \hspace{0.5cm}+\sum_{k=1}^\infty\,\left(\,\delta(\,\alpha+(2\pi k|r_2\xi_2|
+\pi^2\xi_3^2)^{1/2}\,)+\delta(\,\alpha-(2\pi k|r_2\xi_2|
+\pi^2\xi_3^2)^{1/2}\,)\,\right)\,\bigg)\\
m_4 &=& \sum_{\xi_1\neq 0}\sum_{\xi_2\neq 0}
\frac{\gcd(\,|r_1\xi_1|,|r_2\xi_2|\,)}{2}\\
& & \hspace{1cm}\sum_{0\le\xi_3\le 2q_{(\xi_1,\xi_2)}-1}\sum_{\xi_4}
\left(\,\delta(\,\lambda_0(\xi)\,)+\sum_{k=1}^\infty\,
\left(\,\delta(\,\lambda^{+}_k(\xi)\,)+\delta(\,\lambda^{-}_k(\xi)\,)\,
\right)\,\right)
\end{eqnarray*}
In particular, if $\eps_\nu=0$ for $\nu=1$ or $2$, then the numbers
$\{\alpha+(2k+\eps_{2+\nu})\pi :k\in\ZZ\}$ are eigenvalues of $D$ and each
of them has infinite multiplicity.

In this example, the spectrum of $D$ is a non-discrete subset of $\RR$,
no matter which homomorphism $\eps:\Gamma\to\ZZ_2$ defining the underlying
spin structure is chosen. Indeed, $\alpha^\ast:=\alpha+\pi\sgn(r_2)\eps_3$ is
an accumulation point of~$\sigma(D)$. To see this, we consider the sequence
$\xi_n\in\cR_{\eps'}^{(4)}$ given by $\xi_{n1}=2+\eps_1$, $\xi_{n2}=2n+\eps_2$,
$\xi_{n3}=\eps_3$ and $\xi_{n4}=\sgn(r_1r_2)(2+\eps_4)$. Then
$\lambda_0(\xi_n)\neq\alpha^\ast$ and $\lambda_0(\xi_n)\to\alpha^\ast$
for $n\to+\infty$.

\subsection{A three-step nilpotent example}

Let $r_1,r_2\in \ZZ\setminus \{0\}$ be such that $r_1r_2$ is even. Define a Lie
algebra structure on $\fg:=\Span\{e_1,e_2,e_3,b\}$ such that $\fn:=\Span\{e_1,e_2,e_3\}$
is an abelian ideal and $[b,X]=B(X)$ for $X\in\fn$, where $B:\fn\rightarrow\fn$ is
given by  
$$ B=\left( \begin{array}{ccc} 0&r_1&0\\ 0&0&r_2\\ 0&0&0 \end{array}\right)$$ 
with respect to the basis $e_1,e_2,e_3$ of $\fn$. Let $G$ be the simply-connected
Lie group with Lie algebra $\fg$. Then $G=\RR^3\rtimes_A\RR$ with 
$$A(t)=\exp tB =\left( \begin{array}{ccc} 1&tr_1&t^2r_1r_2/2\\ 0&1&tr_2\\ 0&0&1
\end{array}\right).$$ 
Since $A(1)$ is in $SL(2,\ZZ)$, the subset $\Gamma:=\ZZ^3\rtimes_A\ZZ$ a uniform
discrete subgroup of~$G$. Let $(\cH,g)$ be the oriented sub-Riemannian structure
having $s_1:=e_3$, $s_2:=b$ as a positively oriented orthonormal frame. Then $\cH$
is bracket generating.

The spin structures of $\dot{\cH}$ correspond to homomorphisms $\eps:\Gamma
\rightarrow\ZZ_2$. As above we write $\eps(k,l)=\eps'(k)\cdot\dot{\eps}(l)$,
where $\dot{\eps}:\ZZ\rightarrow \ZZ_2$ is an arbitrary homomorphism and
$\eps':\ZZ^3\rightarrow \ZZ_2$ is a homomorphism satisfying~(\ref{EsumA}), which,
in this example, means that $r_1\eps_1$ and $r_1r_2\eps_1/2+r_2\eps_2$ are both even.
More precisely, this shows: If $r_1$ and $r_2$ are both even, then $\eps_1$ and
$\eps_2$ are arbitrary. If $r_1$ is odd and $r_2$ is even, then $\eps_1=0$
and $\eps_2$ is arbitrary. Now suppose that $r_2$ is odd. If, in addition, $r_1$
is odd, then $\eps_1=\eps_2=0$. If $r_1$ is even but not divisble by $4$, then
either $\eps_1=\eps_2=0$ or $\eps_1=\eps_2=1$. Finally, if $r_1$ is divisible
by $4$, then $\eps_2=0$.

Clearly ${\cal V}:=\Span\{e_1,e_2\}$ is a complement of $\cH$ in the tangent bundle $TG$.
Using the projection onto $\cH$ along ${\cal V}$, we define a left-invariant connection
$\nabla$ in $\cH$ by the Koszul formula (\ref{Ediv}). Since $\proj [s_1,s_2]=0$, all
Christoffel symbols $\Gamma_{ij}^k$ vanish. In particular, the sub-Dirac operator is
symmetric and equals
$$D=s_1\cdot \partial_{s_1}+s_2\cdot \partial_{s_2}\,,$$
where we use the simple $\Cl(\cH_e)$-module structure on $\CC^2$ defined by
$$s_1\mapsto\left(\begin{array}{cc} 0&-1\\1&0 \end{array}\right),\quad s_2\mapsto
\left(\begin{array}{cc} i&0\\0&-i \end{array}\right).$$
On the other hand, we have
\begin{equation}\label{EAt} A^\top(t)\xi = \left(\begin{array}{c} \xi_1\\ \xi_2+tr_1\xi_1\\
\xi_3+tr_2\xi_2+t^2 r_1r_2\xi_1/2\end{array}\right).
\end{equation}
In particular, the sets
\begin{eqnarray*}
R^{(1)}&:=& \{\xi\in\RR^3\mid \xi_1=\xi_2=0\},\\ 
R^{(2)}&:=& \{\xi\in\RR^3\mid \xi_1=0,\ \xi_2\not=0\},\\ 
R^{(3)}&:=& \{\xi\in\RR^3\mid \xi_1\not=0\}
\end{eqnarray*}
are invariant under $A^\top(t)$ for all $t\in\RR$.

Let us first consider $D_\theta$ for the orbit $\theta=\{\xi\}$ of an
element $\xi\in R^{(1)}$. Then, according to (\ref{eq:mu_k_for_2dim_H}), the spectrum
of~$D_\theta$ consists of the eigenvalues
$$\mu_k^\pm(\xi)=\pm \pi \left( (2k+\dot{\eps}(1))^2+\xi_3^2\right)^{1/2},\ \  k\in\ZZ.$$

Now consider $\xi\in R^{(2)}$. Then $D_\xi$ has the form 
\begin{equation}\label{Elast}
  \left( \begin{array}{cc} i\partial_t & \bar\omega\\
                   \omega & -i\partial_t
                  \end{array}\right)
\end{equation}
with $\omega(t)=a\omega_1t+\omega_0$, where
$$ a=\pi |r_2\xi_2|,\quad \omega_1=\sgn(r_2\xi_2)\cdot i,\quad \omega_0= \pi i \xi_3.$$
According to (\ref{eq:lambda_for_2dim_H}) the spectrum of $D_\xi$ consists of the
eigenvalues $\lambda_0=0$ and $$\lambda_k^\pm = \pm (2\pi |r_2\xi_2|k)^{1/2},\quad
k\in\NN\setminus\{0\}.$$

Finally, take $\xi\in R^{(3)}$. Then 
$ D_\xi$ is of the form (\ref{Elast})
where $\omega(t)=i\pi (\xi_1r_1r_2t^2/2 +\xi_2r_2t +\xi_3)$. Hence
$$D_\xi^2=\left( \begin{array}{cc} -\partial_t^2-\omega(t)^2 & -i\omega'(t)\\
                   -i\omega'(t) & -\partial_t^2 -\omega(t)^2
                  \end{array}\right).$$
Obviously, $D_\xi^2$ is time-independent diagonalisable. More exactly, $D_\xi^2$ is
conjugate to
$$\left( \begin{array}{cc} -\partial_t^2-\omega(t)^2 -i \omega'(t)&0\\
                  0 & -\partial_t^2 -\omega(t)^2+i\omega'(t)
                  \end{array}\right).$$
The operators $ -\partial_t^2-\omega(t)^2 \mp i\omega'(t)$ are of the form 
$$ P_{a,b,c}^\pm:= \-\partial_t^2+(at^2+bt+c)^2\pm (2at+b)$$
for
$$a=\pi\xi_1r_1r_2\not=0,\quad b=\pi\xi_2r_2,\quad c=\pi\xi_3.$$
We consider the bijection
$$L^2(\RR)\longrightarrow L^2(\RR),\quad \ph\longmapsto \tilde \ph, \ \tilde\ph(t)
=\frac1{x^2}\ph(xt+y),$$
where $x=a^{1/3}$, $y=ba^{-2/3}/2$.

We define $P_c^\pm:=P_{1,0,c}^\pm$. 

{\bf Claim.} The equation $P_{a,b,c}^\pm\tilde\ph=\tilde\lambda\tilde\ph$ is equivalent
to $P_{c_1}^\pm\ph=\lambda\ph$, where
$$c_1=-b^2a^{-4/3}/2 +ca^{-1/3},\quad \tilde \lambda = a^{2/3}\lambda.$$
Indeed, assume that $P_{c_1}^\pm\ph=\lambda\ph$. Then $\ph''(t)=\left( (t^2+c_1)^2\pm2t
-\lambda\right)\ph(t)$ holds. Hence
\begin{eqnarray*}
 \lefteqn{(P_{a,b,c}^\pm\tilde \ph)(t)\ =\ -(\partial_t^2\tilde\ph)(t) 
+ \left( (at^2+bt+c)^2\pm(2at+b)\right)\tilde\ph(t)}\\
&&=\ \left( -x^2\left(((xt+y)^2+c_1)^2\pm2(xt+y)-\lambda\right) 
+ (at^2+bt+c)^2\pm(2at+b)\right)\tilde\ph(t)\\
&& =\ x^2\lambda\tilde\ph(t)\ =\ a^{2/3}\lambda\tilde\ph.
\end{eqnarray*}
The converse can be proven similarly using $\ph(t)=x^2\tilde\ph(t/x-y/x)$.

It is well known that the Schr\"odinger operator $P_c^\pm$ having a polynomial
potential of degree 4 has the following properties \cite{EGS,T}. The spectrum of $P_c^\pm$
is discrete. All eigenvalues are real and simple. They can be arranged into an increasing
sequence $\lambda_0<\lambda_1<\dots \to \infty$ and satisfy
$$\lambda_k\sim \left(\frac{\sqrt\pi\Gamma(7/4)\cdot k}{\Gamma(5/4)}\right)^{4/3}.$$
Obviously, $P_c^+$ and $P_c^-$ have the same eigenvalues. We will denote these eigenvalues
by $\lambda_k(c)$, $k\in\NN$. 

Since $\dim \cH$ is even the spectrum of $D_\xi$ is symmetric. We conclude that
${\rm spec}(D_\xi)$ consists of the eigenvalues
$$\pm \left(a^{2/3}\lambda_k (-4b^2 a^{-4/3}+ca^{-1/3})\right)^{1/2},\quad k\in\NN,$$
where $a=\pi\xi_1r_1r_2/2$, $b=\pi\xi_2r_2$, $c=\pi \xi_3$.

Next we determine a set of representatives of the $\RR$-orbits in $\RR^3$ that
intersect $\Sigma_{\eps'}$ and the number of $\ZZ$-orbits that are contained in them.
Obviously,
$${\cal R}^{(1)}:= R^{(1)}\cap \Sigma_{\eps'}$$
is the set of fixed points in $\Sigma_{\eps'}$ and
$${\cal R}^{(2)}:= \{\xi\in\Sigma_{\eps'}\mid \xi_1=0,\ \xi_2\not=0,\ \xi_3=\eps_3\}$$
is a set of representatives of the $\RR$-orbits in $R^{(2)}$ that interset $\Sigma_{\eps'}$.
For $\xi\in{\cal R}^{(2)}$ the $\RR$-orbit through $\xi$ contains $|r_2\xi_2|/2\,$
$\ZZ$-orbits. Now we turn to orbits contained in $R^{(3)}$. For a given number
$k\in\ZZ\setminus\{0\}$ let $p,q\in \ZZ$, $q>0$ be such that 
$$\frac{|r_2|}{|r_1k|}=\frac pq, \quad (p,q)=1$$
and put $q(k):=q$. Moreover, for $l,q\in\NN$, $q>0$ we define
$$M(l,q):=\{(m_1,m_2)\mid m_1,m_2\in\NN\setminus\{0\},\ m_1+m_2=l,\ q|m_1m_2\}.$$
We will show:
 \begin{enumerate}
\item The set
$$ {\cal R}^{(3)}:=\{\xi\in\Sigma_{\eps'}\mid 0\le\xi_2<|r_1\xi_1|,\ M(\xi_2, q(\xi_1))
=\emptyset\}$$
is a set of representatives of $\RR$-orbits in $R^{(3)}$ that intersect $\Sigma_{\eps'}$.
\item For $\xi \in{\cal R}^{(3)}$ the number of $\ZZ$-orbits contained in the $\RR$-orbit
of $\xi$ equals
$$m(\xi_1,\xi_2):=\#\{k\in\NN\mid \xi_2+2k<|r_1\xi_1|,\ q(\xi_1)|k(k+\xi_2)\}.$$
\end{enumerate}
Take $\xi\in R^{(3)}\cap \Sigma_{\eps'}$ and denote by $\theta$ the $\RR$-orbit of $\xi$.
Using (\ref{EAt}) we see that $A^\top(t)\xi$ is in $\Sigma_{\eps'}$ if and only if
$tr_1\xi_1$ and $t^2r_1r_2\xi_1/2+tr_2\xi_2$ are in $2\ZZ$. The latter condition is
equivalent to 
\begin{equation}\label{Etq}
t=\frac{2k}{r_1\xi_1},\quad q(\xi_1)|k(k+\xi_2)
\end{equation}
for some $k\in\ZZ$. Obviously, we may choose $\hat\xi=(\xi_1,\hat\xi_2,\hat\xi_3)\in\theta$
such that $0\le\hat \xi_2<|r_1\xi_1|$. Now we want to choose $\hat \xi$ is such a way that
$\hat\xi_2\ge 0$ is minimal, which ensures the  uniqueness of the representative. By
(\ref{Etq}), $\hat\xi_2$ is minimal if and only if there does not exist an integer $k$,
$-[\hat\xi_2/2]\le k\le-1$, such that $q(\xi_1)|k(k+\hat\xi_2)$. The latter condition is
equivalent to $q(\xi_1)|(-k)(k+\hat\xi_2)$. Hence $\hat\xi_2$ is minimal if and
only if $\hat\xi_2$ does not decompose as a sum $\hat\xi_2=m_1+m_2$ with
$m_1, m_2\in\NN\setminus\{0\}$ and $q(\xi_1)|m_1m_2$. This proves the first assertion.
The second one follows from (\ref{Etq}).

Now we can give an expression for $m(D)$. In the following sums are taken over
$\xi_i\in\eps_1+2\ZZ$, $i=1,2,3$. Moreover, we willtake another index of summation,
namely $\xi_4\in\dot{\eps}(1)+2\ZZ$. Furthermore, $\kappa\in\{1,-1\}$. Then
\begin{eqnarray*}
\lefteqn{ m(D)= \sum_{\xi_3,\xi_4}\sum_\kappa\delta
\left(\kappa\pi(\xi_3^2+\xi_4^2)^{1/2}\right)}\\
&& +\sum_{\xi_2>0} |r_2\xi_2|\left(\delta(0)+\sum_{k=1}^\infty\sum_\kappa
\delta(\kappa(2\pi k |r_2\xi_2|)^{1/2})\right)\\
&&+\sum_{\xi_1\not=0,\xi_2,\xi_3} \hspace{-9pt}  m(\xi_1,\xi_2)\cdot\sum_{k=0}^\infty
\sum_\kappa\delta\left( \kappa(\pi\xi_1\frac{r_1r_2}2)^{1/3}\lambda_k
\Big((\pi\xi_1\frac{r_1r_2}2)^{-1/3}\pi( \xi_3-\frac{8\xi_2^2r_2}{\xi_1r_1})\Big)^{1/2}\right).
\end{eqnarray*}

\vspace{1.5cm}

Institut f\"ur Mathematik und Informatik\\
Universit\"at Greifswald\\
17489 Greifswald\\
GERMANY\\
email: oungerma@uni-greifswald.de

\end{document}